\documentclass[11pt,oneside]{amsart}
\usepackage{amssymb,amsmath,amsthm,amssymb}
\usepackage{amssymb,amsmath,amsthm,bm,scalerel,stackengine}
\usepackage{color}
\usepackage[normalem]{ulem}
\usepackage[bookmarksopen,bookmarksdepth=3,colorlinks,citecolor=red,pagebackref,hypertexnames=true]{hyperref}
\usepackage[msc-links,nobysame,non-sorted-cites, initials]{amsrefs}
\usepackage[inline]{enumitem}
\usepackage{mathtools}
\usepackage[top=25mm, bottom=25mm, left=20mm, right=20mm]{geometry}
\usepackage{amsfonts,amssymb,amsmath,amsthm,amsxtra}
\usepackage[T1]{fontenc}
\usepackage[utf8]{inputenc}
\usepackage{bm}
\usepackage{mathtools}
\usepackage{dsfont}
\usepackage{mathrsfs}
\usepackage{enumitem}
\allowdisplaybreaks

\usepackage{graphics}
\usepackage{hyperref}
\mathtoolsset{showonlyrefs}

\definecolor{darkblue}{RGB}{0,0,160}

\def\eps{\varepsilon}
\def\d{{\rm d}}
\def\R {\mathbb{R}}
\def\N {\mathbb{N}}

\def\supp {{\mathrm{supp}\,}}
\def\Z {{\mathbb Z}}

\newcommand{\cic}{\bm}

\numberwithin{equation}{section}
\numberwithin{counter2}{section}
\newtheorem{proposition}[subsection]{Proposition}
\newtheorem{theorem}[counter]{Theorem}

\newtheorem{lemma}[subsection]{Lemma}

\theoremstyle{definition}
\newtheorem{definition}[subsection]{Definition}
\newtheorem*{remark*}{Remark}
\newtheorem*{warn*}{A word of warning}

\theoremstyle{plain}

\title[Fractionally modulated discrete Carleson and pointwise Ergodic Theorems]{Fractionally modulated discrete Carleson's Theorem and pointwise Ergodic Theorems along certain curves}
\author{Leonidas Daskalakis} 
\address[Leonidas Daskalakis]{Institute of Mathematics,
Polish Academy of Sciences,
\'Sniadeckich 8,
00-656 Warszawa, Poland}
\email{ldaskalakis@impan.pl}
\author[A. Fragkos]{Anastasios Fragkos}
\address[Anastasios Fragkos]{School of Mathematics\\
Georgia Institute of Technology\\
Atlanta,  GA 30332 USA}
\email{anastasiosfragkos@gatech.edu}

\begin{document}
	\maketitle
	\begin{abstract} For $c\in(1,2)$ we consider the following operators
\[
\mathcal{C}_{c}f(x) \coloneqq \sup_{\lambda \in [-1/2,1/2)}\bigg| \sum_{n \neq 0}f(x-n) \frac{e^{2\pi i\lambda \lfloor |n|^{c} \rfloor}}{n}\bigg|\text{,}\quad 	\mathcal{C}^{\mathsf{sgn}}_{c}f(x) \coloneqq \sup_{\lambda \in [-1/2,1/2)}\bigg| \sum_{n \neq 0}f(x-n) \frac{e^{2\pi i\lambda \mathsf{sign(n)} \lfloor |n|^{c} \rfloor}}{n}\bigg| \text{,}
		\]
and prove that both extend boundedly on $\ell^p(\mathbb{Z})$, $p\in(1,\infty)$. The second main result is establishing almost everywhere pointwise convergence for the following ergodic averages
		\[
		A_Nf(x)\coloneqq\frac{1}{N}\sum_{n=1}^Nf(T^nS^{\lfloor n^c\rfloor}x)\text{,}
		\]
		where $T,S\colon X\to X$ are commuting measure-preserving transformations on a $\sigma$-finite measure space $(X,\mu)$, and $f\in L_{\mu}^p(X)$, $p\in(1,\infty)$. The point of departure for both proofs is the study of exponential sums with phases $\xi_2 \lfloor |n^c|\rfloor+ \xi_1n$ through the use of a simple variant of the circle method.
	\end{abstract}
	\section{Introduction}
	We begin by stating the main theorems of the present work.
\begin{definition}
For every $c\in(1,2)$ and every finitely supported $f\colon\mathbb{Z}\to\mathbb{C}$ we define the $c$-modulated Carleson operators by
\begin{equation}\label{defCc}
\mathcal{C}_{c}f(x) \coloneqq \sup_{\lambda \in [-1/2,1/2)} \bigg| \sum_{n \neq 0}f(x-n) \frac{e^{2\pi i\lambda \lfloor |n|^{c} \rfloor}}{n} \bigg|\text{,}\quad 	\mathcal{C}^{\mathsf{sgn}}_{c}f(x) \coloneqq \sup_{\lambda \in [-1/2,1/2)}\bigg| \sum_{n \neq 0}f(x-n) \frac{e^{2\pi i\lambda \mathsf{sign}(n) \lfloor |n|^{c} \rfloor}}{n}\bigg|\text{.}
\end{equation}
\end{definition}
	\begin{theorem}\label{maintheorem}
		For every $c\in(1,2)$ and $p\in(1,\infty)$ there exists a positive constant $C_{c,p}$ such that
\begin{equation}\label{mainestimate}
\|\mathcal{C}_cf\|_{\ell^p(\mathbb{Z})}\le C_{c,p}\|f\|_{\ell^p(\mathbb{Z})}\text{,}\quad\|\mathcal{C}^{\mathsf{sgn}}_cf\|_{\ell^p(\mathbb{Z})}\le C_{c,p}\|f\|_{\ell^p(\mathbb{Z})}\text{.}
\end{equation}
\end{theorem}
In 1995 Stein \cite{steinquadratic} proved that the purely quadratic Carleson operator defined by
\[
\mathcal{C}_2 f(x) \coloneqq \sup_{\lambda \in \R} \bigg| \int_{\R}f(x-t)\frac{e^{2\pi i \lambda t^2}}{t} \d t\bigg|
\]
is bounded on $L^2(\R)$. Subsequently, Guo \cite{GuoOscillatory} generalized the above result by establishing boundedness on $L^p(\mathbb{R})$ for the following operators
\[
\mathcal{C}^{\mathsf{even}}_{\eps}f(x) \coloneqq \sup_{\lambda \in \R} \bigg| \int_{\R}f(x-t) \frac{e^{2\pi i \lambda |t|^\eps}}{t} \d t \bigg|,\quad \mathcal{C}^{\mathsf{odd}}_{\eps}f(x) \coloneqq \sup_{\lambda \in \R} \bigg| \int_{\R}f(x-t) \frac{e^{2\pi i\lambda \mathsf{sign}(t) |t|^\eps}}{t} \d t \bigg|. \] This line of investigation has attracted attention, and since then, variational bounds have been obtained even for higher dimensional analogues, see for example \cite{Guo2020} or \cite{wan2024sharpvariationalinequalitiesoscillatory}, and the natural question regarding the corresponding discrete analogues for more general phases has been raised.
		
		Lillian Pierce conjectured that the discrete analogue of $\mathcal{C}_2$ is bounded on $\ell^2(\Z).$ The first result in this direction was by Krause and Lacey \cite{KrauseLacey16} where the $\ell^2(\Z)$-boundedness was verified under the restriction that the set of modulations has sufficiently small arithmetic Minkowski dimension. Subsequently, this restriction was dropped, the result was generalized in higher dimensions and in the multiparameter setting \cites{JB,KrauseRoos22,Krause24} with the $\ell^p$ case settled as well, for $p\in(1,\infty)$.
		
		Theorem~$\ref{maintheorem}$ is a certain discrete analogue of Guo's result and the $c$-modulated Carleson operators in \eqref{defCc} can be viewed as intermediate operators between the discrete Carleson operator
\[
f \mapsto \sup_{\lambda \in [-1/2,1/2)}\bigg|\sum_{n \neq 0}f(\cdot-n) \frac{e^{2\pi i \lambda n }}{n} \bigg|
\]
and the discrete quadratic Carleson operator
\[
f \mapsto \sup_{\lambda \in [-1/2,1/2)} \bigg|\sum_{n \neq 0}f(\cdot-n) \frac{e^{2\pi i \lambda n^2}}{n}\bigg|.
\]

Understanding the behavior of the operators \eqref{defCc} naturally leads to the study of exponential sums with phases $\xi_2 \lfloor |n^c|\rfloor+ \xi_1n$. Exponential sums of such flavor are frequently encountered in the context of pointwise convergence problems in ergodic theory and discrete harmonic analysis, and our analysis allowed us to establish the following pointwise ergodic theorem, which is the second main result of the present work.
	\begin{theorem}\label{PET}
		Assume $c\in(1,2)$ and let $(X,\mathcal{B},\mu)$ be a $\sigma$-finite measure space equipped with two invertible measure-preserving transformations $T,S\colon X\to X$ which commute. For every $p\in(1,\infty)$ and every $f\in L^p_{\mu}(X)$ we have that
		\[
		\lim_{N\to\infty}\frac{1}{N}\sum_{n=1}^N f(T^nS^{\lfloor n^c\rfloor}x)\text{ exists for $\mu$-a.e. }x\in X\text{.}
		\] 
	\end{theorem}
After Bourgain's seminal work \cite{bg2,bg3,BET} on pointwise ergodic theorems along polynomial and prime orbits there has been an extensive effort in trying to understand pointwise phenomena in ergodic theory along sparse orbits. The case for the averages taken along the orbits $\big(\lfloor n^c\rfloor\big)_{n\in\mathbb{N}}$ for $c\in(1,\infty)\setminus\mathbb{N}$ and $f\in L^p$, $p\in(1,\infty)$, has been settled in \cite{ncfull} and for exponents close to $1$ one has pointwise convergence even on $L^1$, see \cite{UrbanZienkiewicz2007,Mirek2015}. Theorem~$\ref{PET}$ should be understood as a variant of such results where the averaging is taken along the curve $\big(n,\lfloor n^c\rfloor\big)_{n\in\mathbb{N}}$.
\subsection{Strategy of our proofs.} Here we wish to give a brief description of the strategy of our proofs. Although the proof of Theorem~$\ref{maintheorem}$ is much more intricate than that of Theorem~$\ref{PET}$, the philosophy underlying both treatments is the same.

For the proof of Theorem~$\ref{PET}$ Calder\'on's transference principle and standard considerations suggest it suffices to establish $\ell^p$-estimates for the maximal function corresponding to our ergodic averages  considered on the integer shift system as well as $r$-variation estimates on $\ell^2$ along lacunary sequences for any $r>2$, see Theorem~$\ref{quantestpoint}$.

We begin by approximating the corresponding multipliers by their continuous counterparts, see Proposition~$\ref{multapproxforpetclean}$. To carry out this approximation we employ a simple variant of the circle method where the main contribution comes from all the frequencies in a box around the frequencies $(0,0)$. This is typical for non-polynomial phases since they do not introduce any arithmeticity and their treatment is much more straightforward than the corresponding one for $c\in\mathbb{N}$. The main tools used are van der Corput-type estimates and certain floor-removing techniques relying on a famous Fourier series with uniform bounds on the tail, see for example the proof of Proposition~$\ref{expsummin}$, or page 29 in \cite{LPE}. Notably, usually such considerations impose restrictions on the largeness of $c$, but we manage to get, in some sense, the optimal range $(1,2)$. If one tries to naively extend our approximations for $c\in(1,\infty)\setminus\mathbb{N}$ using van der Corput for higher order derivatives, for example Theorem~8.20 in \cite{IK}, the numerology makes the argument fail since such estimates are much weaker than their continuous counterpart. 

After performing the approximation, one may exploit the continuous counterparts of these estimates which have already been established, see Theorem~1.5 in \cite{jim}, by transfering these estimates appropriately to the discrete setting using a general sampling principle, see Proposition~2.1 from \cite{TransferenceLp} and Proposition~3.1 in \cite{MVV}. Such considerations give the desired $\ell^2$-estimates relatively straightforwardly, and to obtain the $\ell^p$-maximal estimates we use the fact that the trivial estimates hold since we are dealing with averaging operators and Riesz-Thorin interpolation with the $\ell^2$-estimates to conclude.  

The proof of Theorem~$\ref{maintheorem}$ is much more delicate. We start by decomposing our kernel smoothly to dyadic pieces, see \eqref{smoothdyad} and then approximate each piece with its continuous counterpart with an argument analogous to the one discussed above, see Proposition~$\ref{approxcom}$.

Ultimately, the same sampling principle allow us to appropriately transfer the continuous analogues of this estimate, see Theorem~$\ref{Guo}$, although this is not as straightforward as before.  

We will use a Sobolev embedding lemma,  but since we do not have any nontrivial control of the derivatives of our error terms, to make the argument work one needs to discard almost all $\lambda$-frequencies before appealing to this lemma. We show that this is possible using a rather delicate $TT^*$ argument. 

Similarly to the proof of Theorem~$\ref{PET}$, the argument relies on $\ell^2$-techniques and our error terms when addressing the $\ell^p$-case are handled via interpolation with $\ell^2$-estimates.

We have organized the paper as follows: in Section~2 we decompose our kernel to dyadic pieces and establish certain exponential sum estimates used for the employment of the circle method, in Section~3 we perform the circle method, in Section~4 we give a short proof of the pointwise ergodic theorem, in Section~5  we perform the aforementioned elaborate $TT^*$ argument allowing us to discard almost all $\lambda$-frequencies for the proof of Theorem~$\ref{maintheorem}$ and we use the Sobolev embedding lemma reducing our proof to the treatment of the main term of our approximation, and in the final section we discuss how one may use the continuous theory to handle the main term.

\subsection{Short remarks on the nonuniformity of the constants}
Motivated by the atypically large, in some sense optimal for such arguments, range of exponents $c$ for which Theorem~$\ref{maintheorem}$ is established, we wish to make some brief comments on the constant $C_{c,p}$ in the estimates \eqref{mainestimate}. Theorem~$\ref{maintheorem}$ holds for $c\in[1,2]$, with the case $c=1$ corresponding essentially to the $L^p$-boundedness of the Carleson operator and the case $c=2$ to the $\ell^2$-boundness of the discrete quadratic Carleson operator. 

It will be clear that the constant $C_{c,p}$ for Theorem~$\ref{maintheorem}$ guaranteed by our proof blows up as $c\to 1^+$ or $c\to 2^-$. At the expense of making our estimates slightly more cumbersome, we have quantified this nonuniformity for $p=2\footnote{As discussed in the previous subsection the constant $C_{c,p}$ for $p\neq 2$ is obtained through various interpolative arguments with $\ell^2$-estimates, and the blowup will propagate accordingly. For the sake of keeping the exposition more reasonable, and since we believe the blowup \eqref{finalblowup} is far from optimal, we have confined ourselves in quantifying the failure of nonuniformity with respect to $c\in(1,2)$ only for the case $p=2$.}$, and the more precise estimate our argument affords us amounts to 
\begin{equation}\label{finalblowup}
C_{c,2}\lesssim (c-1)^{-\frac{1}{2}}(2-c)^{-\frac{13}{4}}\text{.} 
\end{equation}
where the implied constant is absolute. Both blowups are imposed by the method essentially due to applications of van den Corput lemma, either for trigonometric sums or oscillatory integrals. 

This is not suprising since the proposed method here exploits the range $c\in(1,2)$ in two different ways. On the hand, the so-called lack of arithmetical obstructions for the exponential sums with phases $(\xi_2 \lfloor |n|^c\rfloor+\xi_1 n)_{n\in\{1,\dotsc,N\}}$ for $c\in(1,2)$ allows us to compare the corresponding operators with their continuous counterpart without additional number-theoretic considerations, namely the appearence of certain multi-frequency phenomena and Gauss sums, present for $c=2$. On the other hand, since $c>1$, a certain kind of curvature is introduced, and modulation-invariance phenomena which ought to be addressed for the case $c=1$, yielding the Carleson operator, are avoided.

A rather interesting natural question is whether the constant of Theorem~$\ref{maintheorem}$ can be chosen independently of $c\in[1,2]$. An analysis affording uniformity in the estimates of such strength most likely should not only take into account the way in which arithmetical obstructions appear as $c\to 2^{-}$ in a very quantitative manner, but also establish the continuous analogue of the result, namely, the main result from \cite{GuoOscillatory}, with uniform constants for $c\in[1,2]$, an endeavor closely realated to Carleson's theorem, see also Question~1.4 in \cite{Guo Oscillatory}. Both tasks seem to go beyond the current technology and certainly beyond the methods of the present work. 

\subsection{Notation}If $A,B$ are two non-negative quantities, we write $A\lesssim B$ or $B \gtrsim A$  to denote that there exists an absolute positive constant $C$ such that $A\le C B$. We write $A\simeq B$ to denote that $A\lesssim B$ and $A\gtrsim B$. In the following sections we write $e(x)$ instead of $e^{2\pi i x}$. For every $r\in\mathbb{R}$ we let $\mathbb{N}_{\ge r}\coloneqq \mathbb{N}\cap[r,\infty)$. We denote the torus $\mathbb{R}/\mathbb{Z}$ by $\mathbb{T}$ and as usual we identify it with $[-1/2,1/2)$, so that $\|\xi\|=|\xi|$, where $\|\xi\|=\min\{|\xi-n|:\,n\in\mathbb{Z}\}$. Finally, for any $d\in\mathbb{N}$ and any finitely supported function $f\colon \mathbb{Z}^d\to\mathbb{C}$ we define the Fourier transform of $f$ by
\[
\mathcal{F}_{\mathbb{Z}^d}f(\xi)=\hat{f}(\xi)=\sum_{n\in\mathbb{Z}^d}f(n)e(-\xi\cdot n)\text{,}\quad\xi\in\mathbb{T}^d\text{,}
\] 
and for any $g\in L^1(\mathbb{T}^d)$ we define the inverse Fourier transform by
\[
\mathcal{F}_{\mathbb{Z}^d}^{-1}g(x)=g^{\vee}(x)=\int_{\mathbb{T}^d}g(\xi)e(\xi\cdot x) \d \xi\text{,}\quad x\in\mathbb{Z}^d\text{.} 
\]
By orthogonality for any finitely supported $f\colon\mathbb{Z}^d\to\mathbb{C}$ we get $(\hat{f})^{\vee}=f$. For any $f\in L^1(\mathbb{R}^d)$ we define the Fourier transform and the inverse Fourier transform of $f$ by
\[
\mathcal{F}_{\mathbb{R}^d}[f](\xi)=\int_{\mathbb{R}^d}f(x)e(-\xi\cdot x)\d x
\quad\text{and}\quad\mathcal{F}^{-1}_{\mathbb{R}^d}[f](x)=\int_{\mathbb{R}^d}f(\xi)e(\xi\cdot x)\d\xi\quad\text{ respectively.}
\]
For every $r\in[1,\infty)$, $\mathbb{I}\subseteq \mathbb{N}$ with $|\mathbb{I}|\ge 2$, and every family of complex-valued functions $(a_t(x):\,t\in\mathbb{I})$, we define the $r$-variation seminorms by
\[
V^r(a_t(x):\,t\in\mathbb{I})\coloneqq\sup_{J\in\mathbb{N}}\sup_{\substack{t_0<\dots<t_J\\t_j\in\mathbb{I}}}\bigg(\sum_{j=0}^{J-1}|a_{t_{j+1}}(x)-a_{t_{j}}(x)|^r\bigg)^{1/r}\text{.}
\]
Finally, to make the notation less cumbersome and to treat $\mathcal{C}_c$ and $\mathcal{C}_c^{\mathsf{sgn}}$ simultaneously, we let 
\[ 
[t]_{\mathsf{i}}^c=\mathsf{sign}(t)^{\mathsf{i}}|t|^c=\left\{
\begin{array}{ll}
      |t|^c\text{,} & \text{if }\mathsf{i}=0 \\
      \mathsf{sign}(t)|t|^c\text{,} &\text{if }\mathsf{i}=1\text{.}\\
\end{array} 
\right. 
\]
\section*{Acknowledgment}We would like to thank Michael Lacey and Mariusz Mirek for their constant support and helpful discussions. The authors are also grateful to the referee for their careful reading of the manuscript and useful remarks that led to the improvement of the presentation.
\section{Kernel Decomposition and Exponential Sum Estimates}
For any $j\in\mathbb{N}$, let $\psi_j\in\mathcal{C}^{\infty}(\mathbb{R})$ be such that $\sum_{j\in\mathbb{N}}\psi_j(x)=\frac{1}{x}$ for all $|x|\ge 1$, $\psi_j$ odd and $\supp(\psi_j)\subseteq [-2^{j-1},-2^{j-3}]\cup[2^{j-3},2^{j-1}]$. This can be easily achieved by fixing an even smooth bump function $\varphi$ with $\cic{1}_{[-1/4,1/4]}\le\varphi\le \cic{1}_{[-1/2,1/2]}$ and letting $\psi_j(x)=\frac{1}{x}\big(\varphi(2^{-j}x)-\varphi(2^{-j+1}x)\big)$. Also, for $\mathsf{i} \in \left\{0,1\right\}$ let
\begin{equation}\label{smoothdyad}
	m_j^{\mathsf{i}}(\xi,\lambda):=\sum_{n\neq 0}e(\lambda \mathsf{sign}(n)^{\mathsf{i}} \lfloor |n|^c \rfloor-\xi n)\psi_j(n)\text{,}\quad\lambda,\xi\in\mathbb{T}\text{.} 
\end{equation}
Through standard approximation arguments, one sees that it suffices to establish \eqref{mainestimate} for finitely supported functions and from now on we always work with such functions. Fourier inversion yields
\begin{multline}
\sum_{n \neq 0}f(x-n) \frac{e(\lambda \mathsf{sign}(n)^{\mathsf{i}}  \lfloor |n|^{c} \rfloor )}{n}= \int_{\mathbb T} \widehat{f}(\xi) e(\xi x) \left( \sum_{n \neq 0} \frac{e(\lambda \mathsf{sign}(n)^{\mathsf{i}}  \lfloor |n|^{c} \rfloor- \xi n )}{n} \right) \d  \xi
\\
=\sum_{j \in\mathbb{N}} \int_{\mathbb{T}}\widehat{f}(\xi)e(\xi x)  \left(\sum_{n \neq 0} e(\lambda \mathsf{sign}(n)^{\mathsf{i}}  \lfloor |n|^{c} \rfloor-\xi n ) \psi_j(n)  \right) \d \xi=\sum_{j \in\mathbb{N}} \int_{\mathbb{T}} \widehat{f}(\xi) e(\xi x)m_j^{\mathsf{i}}(\xi,\lambda) \d \xi \text{.}
\end{multline}
We now establish certain exponential sum estimates which allow us to carry out the circle method and specifically treat the minor arcs appropriately.
\begin{proposition}\label{expsummin}
	There exists an absolute positive constant $C$ such that for every $c\in(1,2)$ and all $N,M\in\mathbb{N}$ and $\xi_1,\xi_2\in[-1/2,1/2)\setminus\{0\}$ we have
\begin{multline}\label{estwithc}
	\Big|\sum_{n=1}^Ne(\xi_2\lfloor n^c\rfloor+\xi_1 n)\Big|\le C(2-c)^{-1}(c-1)^{-\frac{1}{2}}\cdot
\\
\cdot\min\big\{\big(1+\log(M)\big)\big(NM^{-1}+N^{\frac{c}{2}}M^{\frac{1}{2}}+N^{1-\frac{c}{2}}|\xi_2|^{-\frac{1}{2}  }\big),|\xi_1|^{-1}(N^{c}|\xi_2|+1)\big\}\text{.}
\end{multline}
\end{proposition}
\begin{proof}
	We begin with the first bound. We estimate dyadic parts of the sum: let $P,P'\in\mathbb{N}_{\ge 2}$ be such that $P'\in[P,2P]$, then
	\[
	\bigg|\sum_{P\le n\le P'}e(\xi_2 \lfloor n^c\rfloor+\xi_1 n) \bigg|=\bigg|\sum_{P\le n\le P'}e(\xi_2 n^c+\xi_1 n)e(-\xi_2\{n^c\}) \bigg|\text{.}
	\] 
	We have that for any $M\in\mathbb{N}$, $0<|x|\le 1/2$ and $y\in\mathbb{R}$ 
	\begin{equation}\label{floorapprox}
		e(-x\{y\})=\sum_{|m|\le M}c_m(x)e(my)+O\bigg(\min\bigg\{1,\frac{1}{M\|y\|}\bigg\}\bigg)\text{,}\quad\text{where}\quad c_m(x)=\frac{1-e(-x)}{2\pi i (x+m)}
	\end{equation}
	and the implied constant is absolute, see page 29 in \cite{LPE}. For convenience, let 
	\[
	g_{M,\xi_2}(n)=e(-\xi_2\{n^{c}\})- \sum_{|m|\le M}c_m(\xi_2)e(mn^c)=O\bigg(\min\bigg\{1,\frac{1}{M\|n^c\|}\bigg\}\bigg)\text{,}
	\]
	and note that
	\begin{multline}\label{Interterms}
	\bigg|\sum_{P\le n\le P'}e (\xi_2 \lfloor n^c\rfloor+\xi_1 n) \bigg|
	\\
		\le\bigg|\sum_{P\le n\le P'}g_{M,\xi_2}(n)e(\xi_2 n^c+\xi_1 n)\bigg|+\bigg|\sum_{P\le n\le P'}e(\xi_2 n^c+\xi_1 n)\sum_{|m|\le M}c_m(\xi_2)e(mn^c)\bigg|\text{.}
	\end{multline}
	Following \cite{LPE}, we define in an identical manner the following quantities
	\[
	U_{P,P'}(t,\xi)=\bigg|\sum_{P\le n\le P'}e(tn^c+\xi n)\bigg|\text{,}\quad\text{for $t\in\mathbb{R}$ and $\xi\in[-1/2,1/2)$}
	\]
	and
	\[
	V_{P,P'}(M)=\sum_{P\le n\le P'}\min\Big\{1,\frac{1}{M\|n^c\|}\Big\}\text{,}\quad\text{for $M\in\mathbb{N}$ }\text{.}
\]
Applying van der Corput lemma, see for example Corollary~8.13 in \cite{IK}, yields that
\begin{equation}
U_{P,P'}(t,\xi)\lesssim P^{\frac{c}{2}}|t|^{\frac{1}{2}}+(c-1)^{-\frac{1}{2}}P^{1-\frac{c}{2}}|t|^{-\frac{1}{2}}\text{,}
\end{equation}
which we may use in an identical manner as in the proof of Lemma~4.4 in \cite{LPE} to obtain
\begin{equation}\label{mindyadic}
V_{P,P'}(M)\lesssim (c-1)^{-\frac{1}{2}}(1+\log M)\big(M^{-1}P+P^{\frac{c}{2}}M^{\frac{1}{2}}\big)\text{.}
\end{equation}
For the first summand in the second line of \eqref{Interterms} we immediately see that
	\[
	\bigg|\sum_{P\le n\le P'}g_{M,\xi_2}(n)e(\xi_2 n^c+\xi_1 n)\bigg|\lesssim V_{P,P'}(M)\lesssim  (c-1)^{-\frac{1}{2}}(1+\log M)(M^{-1}P+P^{\frac{c}{2}}M^{\frac{1}{2}})\text{,}
	\]
and for the second summand we note that
	\begin{multline}
	\bigg|\sum_{P\le n\le P'}e(\xi_2 n^c+\xi_1 n)\sum_{|m|\le M}c_m(\xi_2)e(mn^c)\bigg|=\bigg|\sum_{|m|\le M}c_m(\xi_2)\sum_{P\le n\le P'}e((\xi_2+m) n^c+\xi_1 n)\bigg|
	\\
	\le \sum_{|m|\le M}|c_m(\xi_2)|U_{P,P'}(\xi_2+m,\xi_1)	\le	\bigg(\sum_{|m|\le M}|c_m(\xi_2)|\bigg)\sup_{|\xi_2|\le |t|\le M+1/2}U_{P,P'}(t,\xi_1)
\\
\lesssim(1+\log M)\sup_{|\xi_2|\le |t|\le M+1/2}\Big( P^{\frac{c}{2}}|t|^{\frac{1}{2}}+(c-1)^{-\frac{1}{2}}P^{1-\frac{c}{2}}|t|^{-\frac{1}{2}}\Big)	\lesssim
	(c-1)^{-\frac{1}{2}}(1+\log M)\Big( P^{\frac{c}{2}}M^{\frac{1}{2}}+P^{1-\frac{c}{2}}|\xi_2|^{-\frac{1}{2}}\Big)\text{.}
	\end{multline}
Combining these estimes yields
	\begin{equation}\label{dyadfinest}
	\bigg|\sum_{P\le n\le P'}e (\xi_2 \lfloor n^c\rfloor+\xi_1 n) \bigg|\lesssim (c-1)^{-\frac{1}{2}} (1+\log M)\Big( M^{-1}P+P^{\frac{c}{2}}M^{\frac{1}{2}}+P^{1-\frac{c}{2}}|\xi_2|^{-\frac{1}{2}}\Big)\text{.}
	\end{equation}
	For the full sum note that
	\begin{multline}
	\bigg|\sum_{1\le n\le N}e (\xi_2 \lfloor n^c\rfloor+\xi_1 n) \bigg|\lesssim1+\sum_{l=1}^{\lfloor\log_2N\rfloor}\Big|\sum_{2^l \le n\le \min(2^{l+1}-1,N)}e(\xi_2 \lfloor n^c\rfloor+\xi_1 n)\Big|
	\\
	\lesssim 1+(c-1)^{-\frac{1}{2}}\sum_{l=0}^{\lfloor \log_2N\rfloor } (1+\log M) (M^{-1}2^l+2^{lc/2}M^{1/2}+2^{l(1-c/2)}|\xi_2|^{-1/2})
	\\
	\lesssim
\frac{1}{(2-c)\sqrt{c-1}}\cdot(1+\log M)\big(NM^{-1}+N^{\frac{c}{2}}M^{1/2}+N^{1-\frac{c}{2}}|\xi_2|^{-1/2}\big)\text{,}
	\end{multline}
	where for the final estimate we used that for every $\theta>0$ we have 
\begin{equation}\label{sumwithconstants}
\sum_{l=0}^{\lfloor \log_2N\rfloor }2^{l\theta}=\frac{\big(2^\theta\big)^{\lfloor \log_2N\rfloor}-1}{2^{\theta}-1}\le\frac{N^{\theta}}{e^{\theta\log2}-1}\le\frac{N^{\theta}}{\theta\log(2)}\text{,}\quad\text{since $e^x-1\ge x$.}
\end{equation}

For the second estimate, summation by parts yields
	\begin{multline}\label{fortherest}
	\Big|\sum_{n=1}^Ne(\xi_2\lfloor n^c\rfloor+\xi_1 n)\Big|\le \Big|\sum_{n=1}^Ne(\xi_2\lfloor n^c\rfloor+\xi_1 n)-e(\xi_1 n)\Big|+\Big|\sum_{n=1}^Ne(\xi_1 n)\Big|
	\\
	\lesssim
	\Big|\sum_{n=1}^Ne(\xi_1 n)\big(e(\xi_2\lfloor n^c\rfloor)-1\big)\Big|+|\xi_1|^{-1}
	\\
	\lesssim
	\Big|\sum_{n=1}^Ne(\xi_1 n)\Big|\cdot|\big(e(\xi_2\lfloor N^c\rfloor)-1\big)|+\Big|\sum_{n=1}^{N-1}\sum_{k=1}^ne(\xi_1 k)\big(e(\xi_2\lfloor n^c\rfloor)-e(\xi_2\lfloor (n+1)^c\rfloor)\big)\Big|+|\xi_1|^{-1}
	\\
	\lesssim
	|\xi_1|^{-1}+\sum_{n=2}^{N}|\xi_1|^{-1}|\xi_2|\big((n+1)^c-(n-1)^c\big)\lesssim |\xi_1|^{-1}\big(1+|\xi_2|N^c\big)\text{,}\quad\text{as desired.}
	\end{multline}
\end{proof}
As an immediate corollary we obtain the following estimate.
\begin{proposition}\label{exponentialweighted} There exists an absolute positive constant $C$ such that for every $c\in(1,2)$, $\mathsf{i} \in \left\{0,1\right\}$ and all $j\in\mathbb{N}_{\ge(2-c)^{-1}}$ and $\xi_1,\xi_2\in[-1/2,1/2)\setminus\{0\}$ we have 
\begin{multline}
\Big|\sum_{n \in \Z} e(\xi_2 \mathsf{sign}(n)^{\mathsf{i}}  \lfloor |n|^c \rfloor+\xi_1 n ) \psi_j(n)\Big|
\\
\le C 2^{-j}\min\big\{(c-1)^{-\frac{1}{2}} (2-c)j\big(2^{(\frac{c+1}{3})j}+2^{(1-\frac{c}{2})j}|\xi_2|^{-\frac{1}{2}}\big),|\xi_1|^{-1}(2^{cj}|\xi_2|+1)\big\}
\text{.}
\end{multline}
\end{proposition}
\begin{proof}
Note that for every such $j\in\mathbb{N}$ and $P'\in[2^{j-3},2^{j-1}]$ we may apply the estimate \eqref{dyadfinest} for $M=\lfloor 2^{\frac{2-c}{3}j}\rfloor$ to obtain
\begin{multline}
\Big|\sum_{2^{j-3}\le n\le P'} e(\xi_2 \mathsf{sign}(n)^{\mathsf{i}}  \lfloor |n|^c \rfloor+\xi_1 n )\Big|\lesssim(c-1)^{-\frac{1}{2}} (1+(2-c)j)\Big( \lfloor 2^{\frac{2-c}{3}j}\rfloor^{-1}2^j+2^{\frac{c}{2}j}2^{\frac{2-c}{6}j}+2^{(1-\frac{c}{2})j}|\xi_2|^{-\frac{1}{2}}\Big)
\\
\lesssim (c-1)^{-\frac{1}{2}}(2-c)j\Big(2^{(\frac{c+1}{3})j}+2^{(1-\frac{c}{2})j}|\xi_2|^{-\frac{1}{2}}\Big)\text{,}
\end{multline}
where we used that $\lfloor 2^{\frac{2-c}{3}j}\rfloor\ge 2^{\frac{2-c}{3}j}-1\ge \Big(\frac{2^{\frac{1}{3}}-1}{2^{\frac{1}{3}}}\Big)2^{\frac{2-c}{3}j}$, where the last inequality is true since it is equivalent to
\[
\frac{1}{2^{\frac{1}{3}}}2^{\frac{2-c}{3}j}\ge 1\iff2^{\frac{(2-c)j}{3}}\ge2^{\frac{1}{3}}\text{,}\quad\text{which is true by our assumption.}
\]
Using triangle inequality, breaking the sum above as follows $\big|\sum_{2^{j-3}\le n\le P'}\big|\le\big|\sum_{1\le n\le P'}\big|\le+\big|\sum_{1\le n< 2^{j-3}}\big|$ and using the estimate \eqref{fortherest} yields
\[
\Big|\sum_{2^{j-3}\le n\le P'} e(\xi_2 \mathsf{sign}(n)^{\mathsf{i}}  \lfloor |n|^c \rfloor+\xi_1 n )\Big|\lesssim|\xi_1|^{-1}\big(2^{cj}|\xi_2|+1\big)\text{.}
\]
Combining the two estimates above with Lemma~5.5 from \cite{Demeter} yields the following
\begin{multline}
\Big|\sum_{n\ge 0} e(\xi_2 \mathsf{sign}(n)^{\mathsf{i}}  \lfloor |n|^c \rfloor+\xi_1 n ) \psi_j(n)\Big|
\\
\lesssim 2^{-j}\min\big\{ (c-1)^{-\frac{1}{2}}(2-c)j\big(2^{(\frac{c+1}{3})j}+2^{(1-\frac{c}{2})j}|\xi_2|^{-\frac{1}{2}}\big),|\xi_1|^{-1}(2^{cj}|\xi_2|+1)\big\}
\text{,}
\end{multline}
and a symmetric argument establishes the same estimates for $n<0$. The proof is complete.
\end{proof}
It is clear that in both propositions above we do not have to exclude the cases where only one frequency is $0$. More specifically, if $\xi_2=0$ and $\xi_1\neq 0$, then the second estimate remains true. Similarly, for $\xi_2\neq 0$ and $\xi_1=0$, the first estimate holds.

\section{Major Box and Multiplier Approximations}
The exponential sum estimates allow us to employ a simplified variant of the circle method. The arithmeticity present in the analysis required when dealing with polynomial phases does not enter the picture here. The only
frequencies that are expected to be significant for $m^{\mathsf{i}}_j$ are the ones close to $(0,0)$. More precisely, let us fix one parameter $\varepsilon\coloneqq \frac{2-c}{4}\in(0,\min\{1/4,2-c\})\footnote{Let us note that if one does not wish to keep track of the dependencies of the strength of the estimates with respect to $c\in(1,2)$, as we do for the majority of the estimates in this section, any choice of $\varepsilon\in(0,\min\{1/4,2-c\})$ will be admissible.}$ and define the major box at scale $2^j$ as
\[
\mathfrak{M}_j \coloneqq \big\{ (\xi,\lambda)\in[-1/2,1/2)^2:\,|\xi| \leq 2^{(2\eps-1)j},\,| \lambda| \leq 2^{(\eps-c)j}\big\}\text{.}
\]
We collect some useful lemmas to appropriately approximate our multiplier, see Proposition~$\ref{approxcom}$.
\begin{lemma}\label{L1}
For every $j\in\mathbb{N}$, $\mathsf{i}\in\{0,1\}$, and $(\xi,\lambda)\in\mathfrak{M}_j$ we have that
	\begin{equation}\label{reduction1}
	m_j^{\mathsf{i}}(\xi,\lambda)=\int_{\mathbb{R}}e(\lambda[t]_{\mathsf{i}}^c-\xi t)\psi_j(t)\d t+O\big(2^{(2\varepsilon-1)j}\big)\text{.}
\end{equation}
where the implied constant is absolute.
\end{lemma}
\begin{proof}
Since we have the trivial estimates $\big|m_j^{\mathsf{i}}(\xi,\lambda)\big|\lesssim 1$ and $\Big|\int_{\mathbb{R}}e(\lambda[t]_{\mathsf{i}}^c-\xi t)\psi_j(t)\d t\Big|\lesssim 1$, it suffices to establish \eqref{reduction1} for $j\ge 3$, and we have
\[
m_j^{\mathsf{i}}(\xi,\lambda)-\int_{\mathbb{R}}e(\lambda[t]_{\mathsf{i}}^c-\xi t)\psi_j(t)\d t=\sum_{n\neq0}\int_n^{n+1}\big(e(\lambda \mathsf{sign}^{\mathsf{i}}(n) \lfloor |n|^c\rfloor-\xi n)\psi_j(n)-e(\lambda[t]_{\mathsf{i}}^c-\xi t)\psi_j(t)\big)\d t\text{.}
\]
Note that for $n\ge 1$ and $t\in[n,n+1]$ we have
\[
|e(\lambda |t|^c-\xi t)-e(\lambda \lfloor n^c \rfloor-\xi n)| \lesssim  |\lambda| \big|t^c- \lfloor n^c \rfloor\big|+|\xi| \big|t-n\big| \lesssim 2^{(\varepsilon-c)j}n^{c-1}+2^{(2\varepsilon-1)j}\text{.}
\]
Thus
\[
\begin{split}
\left|e(\lambda \lfloor |n|^{c} \rfloor-\xi n ) \psi_j(n)-e(\lambda |t|^c-\xi t)\psi_j(t)\right| & \leq \left| \left(e(\lambda |t|^c-\xi t)-e(\lambda \lfloor n^c \rfloor-\xi n)\right) \psi_j(n) \right|  \\ & + \left| e(\lambda |t|^c-\xi t)\psi_j(n)-e(\lambda |t|^c-\xi t) \psi_j(t) \right| \\ & \lesssim \big(2^{(\varepsilon-c)j}n^{c-1}+2^{(2\varepsilon-1)j}\big)|\psi_j(n)|+|\psi_j(n)-\psi_j(t)|\text{.}\end{split}
\]
	For the first summand we get
	\[
	|\psi_j(n)|\big(2^{(\varepsilon-c)j}n^{c-1}+2^{(2\varepsilon-1)j}\big)\lesssim 2^{-j} \big(2^{(\varepsilon-c)j}2^{(c-1)j}+2^{(2\varepsilon-1)j}\big)=2^{-j} \big(2^{(\varepsilon-1)j}+2^{(2\varepsilon-1)j}\big)\lesssim 2^{-j}2^{(2\varepsilon-1)j}\text{,}
	\]
	and for the second one, by the Mean Value Theorem we get
	\[
	|\psi_j(n)-\psi_j(t)|\le\|\psi_j'\|_{L^{\infty}}\lesssim 2^{-2j}\lesssim 2^{-j}2^{(2\varepsilon-1)j}\text{,}
	\]
	where the second to last estimate can be derived straightforwardly, using for example the definition of $\psi_j$.
	
	An identical treatment can be used for $n\le -1$, and putting everything together yields
	\[
	\Big|m_j^{ \mathsf{i}}(\xi,\lambda)-\int_{\mathbb{R}}e(\lambda[t]_{\mathsf{i}}^c-\xi t)\psi_j(t)dt\Big|\lesssim \sum_{2^{j-4}\le|n|\le 2^j}2^{-j}2^{(2\varepsilon-1)j}\lesssim 2^{(2\varepsilon-1)j}\text{,}\quad\text{as desired.}
	\]
\end{proof}
\begin{lemma}\label{minorarcestimate}For every $j\in\mathbb{N}$, $\mathsf{i}\in\{0,1\}$, and $(\xi,\lambda)\in[-1/2,1/2)^2\setminus\mathfrak{M}_j$ we have that
\[
|m_{j}^{\mathsf{i}}(\xi,\lambda)| \lesssim (c-1)^{-\frac{1}{2}}  2^{-\frac{\varepsilon}{4} j}\text{.}
\]
\end{lemma}
\begin{proof}We can assume that $j \varepsilon\ge 1$ because for $j$ such that $j \varepsilon<1$ the estimate above is immediate from the trivial estimate $|m_{j}^{\mathsf{i}}(\xi,\lambda)| \lesssim 1$. Thus Proposition~$\ref{exponentialweighted}$ is applicable and for $(\xi,\lambda)\notin \mathfrak{M}_j$ we get 
\[
|m_j^{\mathsf{i}}(\xi,\lambda)|\lesssim 2^{-j}\min\big\{ (c-1)^{-\frac{1}{2}}(2-c)j\big(2^{(\frac{c+1}{3})j}+2^{(1-\frac{c}{2})j}|\lambda|^{-\frac{1}{2}}\big),|\xi|^{-1}(2^{cj}|\lambda|+1)\big\}\text{.}
	\]
If $|\lambda|>2^{(\varepsilon-c)j}$, then
\[
\frac{|m_j^{\mathsf{i}}(\xi,\lambda)|}{(c-1)^{-\frac{1}{2}}}\lesssim (2-c)j2^{(\frac{c-2}{3})j}+(2-c)j2^{-\frac{c}{2}j}2^{(-\frac{\varepsilon}{2}+\frac{c}{2})j}\lesssim \varepsilon j2^{-\frac{\varepsilon}{2}j}\lesssim2^{-\frac{\varepsilon}{4}j}\text{,}
\]
where we used that $\varepsilon=\frac{2-c}{4}$, and the fact that
\begin{equation}\label{cleantrick}
 x 2^{- \frac{x}{2}}\lesssim 2^{-\frac{x}{4}}\text{.}
\end{equation}
On the other hand, if $|\lambda|\le 2^{(\varepsilon-c)j}$, since $(\xi,\lambda)\notin \mathfrak{M}_j$,  we get that $|\xi|>2^{(2\varepsilon-1)j}$, and thus
\[
|m_j^{\mathsf{i}}(\xi,\lambda)|\lesssim 2^{-j}2^{(1-2\varepsilon)j}(2^{cj}2^{(\varepsilon-c)j}+1)\lesssim 2^{-\varepsilon j}+2^{-2\varepsilon j}\lesssim (c-1)^{-\frac{1}{2}}2^{-\frac{\varepsilon}{4}j}\text{.}
\]
The proof is complete.
\end{proof}
\begin{lemma}\label{continuousvdc} Let $\mathsf{i}\in\{0,1\}$. If $\lambda,\xi\in\mathbb{R}$ are such that $
	|\lambda|>2^{(\varepsilon-c)j}$ or $|\xi|>2^{(2\varepsilon-1)j}$,
	then
	\begin{equation}\label{goalhere}
		\Big| \int_{\mathbb{R}}e(\lambda [t]_{\mathsf{i}}^c-\xi t) \psi_j(t) \d t  \Big|\lesssim (c-1)^{-\frac{1}{2}}2^{-\frac{\varepsilon }{2}j}\text{.}
	\end{equation}
\end{lemma}
\begin{proof}
Using the trivial estimate $\Big| \int_{\mathbb{R}}e(\lambda [t]_{\mathsf{i}}^c-\xi t) \psi_j(t) dt  \Big|\lesssim 1$, we see that it suffices to establish \eqref{goalhere} for $\varepsilon j\gtrsim1$, in particular we prove it below for $j\ge 2\varepsilon^{-1}$. Taking into account the support of $\psi_j$ we may split the integral into the intervals $(-2^{j-1},-2^{j-3})$ and $(2^{j-3},2^{j-1})$ and apply van der Corput's lemma. Without loss of generality, we may assume that $\lambda\ge 0$. In particular, if $|\lambda|>  2^{(\varepsilon-c)j}$ by the second derivative test, see for example Corollary~2.6.8. in \cite{Grafakos}, one has
\[
\Big|\int_{\mathbb{R}}e(\lambda [t]_{\mathsf{i}}^c-\xi t) \psi_j(t) \d t\Big|\lesssim (c(c-1))^{-\frac{1}{2}}  |\lambda|^{-\frac{1}{2}}2^{\frac{2-c}{2}j}2^{-j}\lesssim(c-1)^{-\frac{1}{2}} 2^{(-\frac{\varepsilon}{2}+\frac{c}{2})j}2^{-\frac{c}{2}j}\lesssim(c-1)^{-\frac{1}{2}} 2^{-\frac{\varepsilon }{2}j}\text{.}
	\]
On the other hand, if $|\lambda| \le 2^{(\varepsilon-c)j}$, then $|\xi| > 2^{(2\varepsilon-1)j}$, and we apply the first derivative test. We do this for the interval of positive numbers but the argument is identical for the symmetric one. We start with $\lambda\neq0$, for all $t\in[2^{j-3},2^{j-1}]$ we have that
	\[
	|(\lambda t^c-\xi t)'|=|\lambda c t^{c-1}-\xi|\ge |\xi|-|\lambda c t^{c-1}|\ge |\xi|-2^{(\varepsilon-c)j}c 2^{(c-1)j}>2^{(2\varepsilon-1)j}-c2^{(\varepsilon-1)j}
	\]
	\[
	=2^{(2\varepsilon-1)j}(1-c2^{-\varepsilon j})\ge \frac{1}{2}2^{(2\varepsilon-1)j}\text{,}
	\]
	since $1-c2^{-\varepsilon j}\ge 1/2\iff j\ge \varepsilon^{-1}\log_2(2c)$, which is true because $j\ge 2\varepsilon^{-1}\ge \varepsilon^{-1}\log_2(2c)$.  We have that $(\lambda t^c-\xi t)''=\lambda c (c-1)t^{c-2}>0$ and thus the first derivative is monotone, and van der Corput applicable. We get that
	\[
	\Big|\int_{\mathbb{R}}e(\lambda [t]_{\mathsf{i}}^c-\xi t) \psi_j(t) \d t\Big|\lesssim 2^{(1-2\varepsilon)j}2^{-j}\lesssim 2^{-\frac{\varepsilon}{2} j}\text{.}
	\]
Finally, for $\lambda=0$, using the uniform Schwartz decay of $\psi_j$'s we obtain
\[
\bigg|	\int_{\mathbb{R}}e(-\xi t) \psi_j(t) \d t\bigg| \lesssim |\xi 2^{j}|^{-1} \lesssim 2^{-\frac{\varepsilon}{2}j}\text{,}\quad\text{and the proof is complete.}
\] 
\end{proof}
\begin{lemma}\label{derivtrivial}
For every $\mathsf{i}\in\{0,1\}$ and $j\in\mathbb{N}$ we have that
	\[
	|\partial_{\lambda}m_j^{\mathsf{i}}(\xi,\lambda)|\lesssim 2^{cj}\quad\text{and}\quad\Big|\partial_{\lambda}\Big(\int_{\mathbb{R}}e(\lambda[t]_{\mathsf{i}}^c-\xi t)\psi_j(t)\d t\Big)\Big|\lesssim 2^{cj}\text{.}
	\]
\end{lemma}
\begin{proof}
	This is obvious since 
	\[
	|\partial_{\lambda}m_j(\xi,\lambda)|=\Big|\sum_{2^{j-3}\le |n|\le 2^{j-1}}2\pi i  \mathsf{sign}(n)^{\mathsf{i}} \lfloor |n|^c \rfloor e(\lambda\mathsf{sign}(n)^{\mathsf{i}}\lfloor |n|^c \rfloor-\xi n)\psi_j(n)\Big|\lesssim 2^{cj}\text{,}
	\]
	and by the dominated convergence theorem we get
	\[
	\Big|\partial_{\lambda}\Big(\int_{\mathbb{R}}e(\lambda[t]_{\mathsf{i}}^c-\xi t)\psi_j(t)\d t\Big)\Big|=\Big|\int_{2^{j-3}\le|t|\le 2^{j-1}}2\pi i[t]_{\mathsf{i}}^c e(\lambda[t]_{\mathsf{i}}^c-\xi t)\psi_j(t)\d t\Big|\lesssim 2^{cj}\text{.}
	\]
\end{proof}
We now combine the previous four lemmas to establish the following proposition.
\begin{proposition}\label{approxcom}Fix $\eta_1\in\mathcal{C}^{\infty}_c(\mathbb{R})$ such that $\cic{1}_{[-1/4,1/4]}\le\eta_1\le \cic{1}_{[-1/2,1/2]}$ and let $\eta(\xi,\lambda)=\eta_1(\xi)\eta_1(\lambda)$. For every $j\in\mathbb{N}$ and $\mathsf{i}\in\{0,1\}$, let
	\[
	H_j^{\mathsf{i}}(\xi,\lambda):=\eta(\xi,\lambda)\int_{\mathbb{R}}e(\lambda [t]_{\mathsf{i}}^c-\xi t)\psi_j(t)\d t\quad\text{and}\quad E_j^{\mathsf{i}}:=m^{\mathsf{i}}_j-H_j^{\mathsf{i}}\text{.}
	\]
	Then we have
	\[
	\sup_{\xi,\lambda\in[-1/2,1/2)}|E_j^{\mathsf{i}}(\xi,\lambda)|\lesssim(c-1)^{-\frac{1}{2}} 2^{-\frac{\varepsilon}{4} j}\quad\text{and}\quad\sup_{\xi,\lambda\in[-1/2,1/2)}|\partial_{\lambda} E_j^{\mathsf{i}}(\xi,\lambda)|\lesssim  2^{cj}\text{.}
	\]
\end{proposition}
\begin{proof}
	A similar reasoning with the one in the beginning of the proof of Lemma~$\ref{continuousvdc}$ shows that it suffices to establish the first estimate for $j\ge3$ and for all such $j$ we get that $\mathfrak{M}_j\subseteq [-1/4,1/4]^2$ and thus if $(\xi,\lambda)\in\mathfrak{M}_j$, then $\eta(\xi,\lambda)=1$. 
	The first estimate for $(\xi,\lambda)\in\mathfrak{M}_j$ is therefore immediate from Lemma~$\ref{L1}$, since for all such $(\xi,\lambda)$ we have $|m^{\mathsf{i}}_j(\xi,\lambda)-H^{\mathsf{i}}_j(\xi,\lambda)|\lesssim 2^{(2\varepsilon-1)j}\le 2^{-\frac{\varepsilon}{4} j}$, and to see why the last estimate holds simply note
	\begin{equation}\label{epsilons}
		2\varepsilon-1<-\frac{\varepsilon}{4}\iff\varepsilon<\frac{4}{9}\text{,}\quad\text{which is true since $\varepsilon<1/4$.}
	\end{equation}
	If $(\xi,\lambda)\notin\mathfrak{M}_j$, then Lemmas~$\ref{minorarcestimate}$,$\ref{continuousvdc}$ give that 
\[
|m^{\mathsf{i}}_j(\xi,\lambda)-H^{\mathsf{i}}_j(\xi,\lambda)|\le |m^{\mathsf{i}}_j(\xi,\lambda)|+|H^{\mathsf{i}}_j(\xi,\lambda)|\lesssim(c-1)^{-\frac{1}{2}}2^{-\frac{\varepsilon}{4} j}\text{.}
\]
Therefore we have shown that
	\[
	\sup_{\xi,\lambda\in[-1/2,1/2)}|E_j^{\mathsf{i}}(\xi,\lambda)|\lesssim (c-1)^{-\frac{1}{2}}2^{-\frac{\varepsilon}{4} j}\text{,}
	\]
and Lemma~$\ref{derivtrivial}$ immediately yields the latter assertion of the proposition concluding the proof.
\end{proof}
The afforementioned approximation will play a crucial role in establishing Theorem~$\ref{maintheorem}$. We now proceed in stating and proving the corresponding necessary approximation for the proof of Theorem~$\ref{PET}$. The key ingredient of our argument for establishing the pointwise ergodic theorem is to perform a similar analysis to the one of Proposition~$\ref{approxcom}$ in order to approximate the following multiplier 
\[
k_t(\xi_1,\xi_2)\coloneqq\frac{1}{t}\sum_{1\le n\le t}e(-\xi_2\lfloor n^c\rfloor-\xi_1 n)
\]
by its continuous counterpart. More precisely, fix $c\in(1,2)$, $\tau\in(0,\min\{1/4,2-c\})$ and for every $t> 1$ let 
\[
\widetilde{\mathfrak{M}}_{t}\coloneqq \big\{ (\xi_1,\xi_2)\in[-1/2,1/2)^2:\,|\xi_1| \leq t^{2\tau-1},\,| \xi_2| \leq t^{\tau-c}\big\}\text{.}
\]
We have the following proposition.
\begin{proposition}\label{multapproxforpetclean}
	For every $t\in\mathbb{N}$, let $\mathcal{L}_t(\xi_1,\xi_2):=\frac{1}{t}\int_{0}^{t}e(-\xi_2s^c-\xi_1 s)\d s$.
	There exists a positive constant $C=C(c,\tau)$ such that
	\begin{equation}\label{goalPet}
		\sup_{\xi_1,\xi_2\in[-1/2,1/2)}|k_t(\xi_1,\xi_2)-\mathcal{L}_t(\xi_1,\xi_2)| \le C t^{-\frac{\tau}{4}}\text{,} 
	\end{equation}
	and more precisely,
	\begin{equation}\label{goal1Pet}
		\sup_{(\xi_1,\xi_2)\in\widetilde{\mathfrak{M}}_t}|k_t(\xi_1,\xi_2)-\mathcal{L}_t(\xi_1,\xi_2)| \le C t^{-\frac{\tau}{4}}\text{,}
	\end{equation}
	\begin{equation}
		\sup_{(\xi_1,\xi_2)\in[-1/2,1/2)^2\setminus\widetilde{\mathfrak{M}}_t}|k_t(\xi_1,\xi_2)|\le C t^{-\frac{\tau}{4}}
		\text{,}\quad\text{and}\quad
		\sup_{(\xi_1,\xi_2)\in\mathbb{R}^2\setminus\widetilde{\mathfrak{M}}_t}|\mathcal{L}_t(\xi_1,\xi_2)|\le C t^{-\frac{\tau}{4}}\text{.}
	\end{equation}
\end{proposition}
\begin{proof}
	We fix $c$ and $\tau$ as before and all implied constants here may depend on them. It suffices to establish the estimates for $t\gtrsim_{c,\tau} 1$, and in particular, we will prove the result for $t\ge(c2^c)^{1/\tau}$. For the first assertion of \eqref{goal1Pet} note that for all $(\xi_1,\xi_2)\in\widetilde{\mathfrak{M}}_t$ we have that
	\begin{multline}
	t|k_t(\xi_1,\xi_2)-\mathcal{L}_t(\xi_1,\xi_2)|=\bigg|\sum_{m=1}^{t}\int_{m-1}^{m}\big(e(-\xi_1 m-\xi_2\lfloor m^c\rfloor)-e(-\xi_1 s-\xi_2s^c)\big)\d s\bigg|
	\\
	\le\sum_{m=1}^{t}\int_{m-1}^m\big|e(-\xi_1 m-\xi_2\lfloor m^c\rfloor)-e(-\xi_1 s-\xi_2s^c)\big|\d s
\\
	\lesssim\sum_{m=1}^{t}\int_{m-1}^m\Big(|\xi_1|(m-s)+|\xi_2|(\lfloor m^c\rfloor- s^c)\Big)ds\lesssim \sum_{m=1}^{t}\big(t^{2\tau-1}+t^{\tau-c}m^{c-1}\big)\lesssim
	t^{2\tau}+t^{\tau}\lesssim t^{2\tau}\text{.}
	\end{multline}
	We have used that for all $s\in[m-1,m]$, we have $\lfloor m^c\rfloor- s^c\le m^c-(m-1)^c\lesssim m^{c-1}$, and the last estimate can be justified by the Mean Value Theorem. This immediately implies that 
	\begin{equation}\label{goal1done}		\sup_{\xi\in\widetilde{\mathfrak{M}}_t}|k_t(\xi)-\mathcal{L}_t(\xi)|\lesssim t^{2\tau-1}\le t^{-\frac{\tau}{4}}\text{,}\quad\text{where the last inequality is justified as in \eqref{epsilons}.}
	\end{equation}
Now we establish the other two estimates in two steps. Let $\widetilde{\mathfrak{M}}_t^{\mathsf{c}}=[-1/2,1/2)\setminus\widetilde{\mathfrak{M}}_t$ and let $(\xi_1,\xi_2)\in \widetilde{\mathfrak{M}}_t^{\mathsf{c}}$. If $|\xi_2|>t^{\tau-c}$, then by Proposition~$\ref{expsummin}$ applied for $N=t$ and $M=\lfloor N^{\frac{2-c}{3}}\rfloor$ we get
	\[
	|k_t(\xi_1,\xi_2)|\lesssim (1+\log t)(t^{\frac{c-2}{3}}+t^{-\frac{c}{2}}t^{\frac{c-\tau}{2}})\lesssim t^{-\frac{\tau}{4}}\text{,}\quad \text{since $\tau<2-c$.}
	\]
	If $|\xi_2|\le t^{\tau-c}$, then since $(\xi_1,\xi_2)\in\widetilde{\mathfrak{M}}^{\mathsf{c}}_t$, we must have that $|\xi_1|>t^{2\tau-1}$. But then by the second estimate of Proposition~$\ref{expsummin}$ we obtain
	\begin{equation}
		\label{goal22half}
		|k_t(\xi_1,\xi_2)|\lesssim t^{-1}|\xi_1|^{-1}(1+t^c|\xi_2|)\lesssim t^{-1}t^{1-2\tau}(1+t^ct^{\tau-c})\lesssim t^{-2\tau }+t^{-\tau}\lesssim t^{-\frac{\tau}{4}}\text{.} 
	\end{equation}
	Thus we have shown that $\sup_{\xi\in\widetilde{\mathfrak{M}}^\mathsf{c}_t}|k_t(\xi)|\lesssim t^{-\frac{\tau}{4}}$. It remains to show that $\sup_{\xi\in\mathbb{R}^2\setminus\widetilde{\mathfrak{M}}_t}|\mathcal{L}_t(\xi)|\lesssim t^{-\frac{\tau}{4}}$ and we achieve this via a similar reasoning. Let $(\xi_1,\xi_2)\in \mathbb{R}^2\setminus\widetilde{\mathfrak{M}}_t$. If $|\xi_2|>t^{\tau-c}$, then note that
	\begin{equation}\label{dyadpartition}
		|\mathcal{L}_t(\xi_1,\xi_2)|\le t^{-1}\Big|\int_{0}^{t}e(-\xi_2s^c-\xi_1 s)\d s\Big|\le O(t^{-1})+t^{-1}\sum_{m=0}^{\lfloor\log_2 t\rfloor}\Big|\int_{2^m}^{\min\{2^{m+1},t\}}e(-\xi_2s^c-\xi_1 s)\d s\Big|\text{.}
	\end{equation}
	For any $m\in\{0,\dotsc,\lfloor \log_2 t\rfloor\}$ and $s\in[2^m,2^{m+1}]$ we have that the second derivative of the phase can be estimated as follows  
	\[
	|(-\xi_2s^c-\xi_1 s)''|\gtrsim|\xi_2|s^{c-2}\gtrsim |\xi_2|2^{m(c-2)}\text{,}
	\]
	and we may apply  van der Corput to obtain
	\[
	\Big|\int_{2^m}^{\min\{2^{m+1},t\}}e(-\xi_2s^c-\xi_1 s)\d s\Big|\lesssim |\xi_2|^{-1/2}2^{\frac{2-c}{2}m}\lesssim t^{\frac{c-\tau}{2}}2^{\frac{2-c}{2}m} \text{.}
	\]
	Applying this estimate to \eqref{dyadpartition} yields
	\[
	|\mathcal{L}_t(\xi_1,\xi_2)|\lesssim t^{-1}+t^{-1}\sum_{m=0}^{\lfloor \log_2t\rfloor}t^{\frac{c-\tau}{2}}2^{\frac{2-c}{2}m}=\lesssim t^{-1}+t^{\frac{c-\tau}{2}-1}\sum_{m=0}^{\lfloor \log_2t\rfloor}\big(2^{\frac{2-c}{2}}\big)^m\lesssim t^{-1}+t^{\frac{c-\tau}{2}-1}t^{\frac{2-c}{2}}\lesssim t^{-\frac{\tau}{4}} \text{.}
	\]
	If $|\xi_2|\le t^{\tau-c}$, then since $(\xi_1,\xi_2)\in\mathbb{R}^2\setminus\widetilde{\mathfrak{M}}_t$, we have that $|\xi_1|>t^{2\tau-1}$. We will use the first derivative test here. For any $m\in\{0,\dotsc,\lfloor \log_2 t\rfloor\}$ and $s\in[2^m,2^{m+1}]$ we have that
	\[
	|(-\xi_2s^c-\xi_1s)'|\ge|\xi_1|-|c\xi_2s^{c-1}|> t^{2\tau-1}-ct^{\tau-c}2^{(m+1)(c-1)}\ge \frac{1}{2}t^{2\tau-1}\text{,}
	\]
	where for the last estimate we used the fact that $t\ge(c2^c)^{1/\tau}$. More precisely
	\[
	t^{2\tau-1}-ct^{\tau-c}2^{(m+1)(c-1)}\ge t^{2\tau-1}\big(1-c2^{c-1}t^{-\tau}\big)\ge \frac{1}{2}t^{2\tau-1}\text{,}
	\]
	since $	1-c2^{c-1}t^{-\tau}\ge 1/2\iff t\ge(c2^c)^{1/\tau}$.
	If $\xi_2\neq 0$, then the first derivative of the phase is monotonic and by van der Corput we obtain
	\[
	|\mathcal{L}_t(\xi_1,\xi_2)|\lesssim t^{-1}+t^{-1}
	\sum_{m=0}^{\lfloor \log_2 t\rfloor}\Big|\int_{2^m}^{\min\{2^{m+1},t\}}e(-\xi_2s^c-\xi_1 s)\d s\Big|\lesssim t^{-1}+t^{-1}
	\sum_{m=0}^{\lfloor \log_2 t\rfloor}t^{1-2\tau}\lesssim t^{-\frac{\tau}{4}}\text{.}
	\]
	If $\xi_2=0$ the estimate is trivial since
	\begin{equation}\label{done}
		|\mathcal{L}_t(\xi_1,\xi_2)|= t^{-1}\Big|\int_{0}^{t}e(-\xi_1 s)ds\Big|\lesssim t^{-1}|\xi_1|^{-1}\lesssim t^{-2\tau}\lesssim t^{-\frac{\tau}{4}}\text{.}
	\end{equation}
	We have shown that $\sup_{\xi\in\mathbb{R}^2\setminus\widetilde{\mathfrak{M}}_t}|\mathcal{L}_t(\xi)|\lesssim t^{-\frac{\tau}{4}}$ and the proof is complete.
\end{proof} 
\section{Pointwise Ergodic Theorem}
In this section we give a short proof of Theorem~$\ref{PET}$. For every $t\in\mathbb{N}$ and $f\colon\mathbb{Z}^2\to\mathbb{C}$ let 
\[
A_tf(x_1,x_2)\coloneqq\frac{1}{t}\sum_{n\le t}f(x_1-n,x_2-\lfloor n^c\rfloor)\text{.}
\]
By Calder\'on's transference principle and standard considerations it suffices to establish the following.
\begin{theorem}\label{quantestpoint}Assume $c\in(1,2)$, $p\in(1,\infty)$, $\lambda\in(1,\infty)$ and $r\in(2,\infty)$. Then there exist positive constants $C=C(c,p)$ and $C'=C'(c,p,\lambda,r)$ such that
	\[
	\big\|\sup_{t\in\mathbb{N}}|A_tf|\big\|_{\ell^p(\mathbb{Z}^2)}\le C\|f\|_{\ell^p(\mathbb{Z}^2)}
	\quad\text{and}
	\quad
	\|V^r(A_{\lfloor \lambda^n\rfloor}f:\,n\in\mathbb{N})\|_{\ell^2(\mathbb{Z}^2)}\le C' \|f\|_{\ell^2(\mathbb{Z}^2)}\text{.}
	\]
\end{theorem}
Before proceeding to the proof of Theorem~$\ref{quantestpoint}$ we state the corresponding result for the continuous counterpart of our averaging operator. For every $t\in[1,\infty)$ and $g\colon\mathbb{R}^2\to\mathbb{C}$ let 
\[
B_tg(x_1,x_2)\coloneqq\frac{1}{t}\int_{0}^tf(x_1-s,x_2- s^c)\d s\text{.}
\]
\begin{theorem}\label{contcounterpart} For any $p\in(1,\infty)$ and $r\in(2,\infty)$ there exists positive constants $C=C(p)$ and $C'=C(p,r)$ such that
	\begin{equation}\label{maxcontt}
		\big\|\sup_{t\in\mathbb{N}}|B_tf|\big\|_{L^p(\mathbb{R}^2)}\le C\|f\|_{L^p(\mathbb{R}^2)}
	\end{equation}
	and
	\begin{equation}\label{rvarcont}
		\big\|V^r\big(B_tf:\,t\in\mathbb{N}\big)\big\|_{L^p(\mathbb{R}^2)}\le C'\|f\|_{L^p(\mathbb{R}^2)}
		\text{.}
	\end{equation}
\end{theorem}
\begin{proof}The proof of the estimate \eqref{rvarcont} can be found in \cite{jim}, see Theorem~1.5, and it is easy to see that it implies \eqref{maxcontt} since we have the pointwise estimate $\sup_{t\in\mathbb{N}}|B_tf(x)|\le |B_1f(x)|+V^4\big(B_tf(x):\,t\in\mathbb{N}\big)$, which in turn implies that $\big\|\sup_{t\in\mathbb{N}}|B_tf|\big\|_{L^p(\mathbb{R}^2)}\lesssim \|f\|_{L^p(\mathbb{R}^2)}+\big\|V^4\big(B_tf(x):\,t\in\mathbb{N}\big)\|_{L^p(\mathbb{R}^2)}$. 
\end{proof}
We are now ready to give a short proof of Theorem~$\ref{quantestpoint}$.
\begin{proof}[Proof of Theorem~$\ref{quantestpoint}$]Again we fix $c$ and $\tau\in(0,\min\{1/4,2-c\})$ as before and all  implied constants may depend on them. We begin with the maximal estimate. Firstly, note that $\sup_{t\in\mathbb{N}}|A_tf|\lesssim \sup_{n\in\mathbb{N}_0}A_{2^n}|f|$. Let $\eta\in\mathcal{C}_c^{\infty}(\mathbb{R}^2)$ be such that $\cic{1}_{[-1/4,1/4)^2}\le\eta\le\cic{1}_{[-1/2,1/2)^2}$, and let $\mathcal{L}_t(\xi_1,\xi_2):=\frac{1}{t}\int_{0}^{t}e(-\xi_2s^c-\xi_1 s)\d s$. Fix $n_0\in\mathbb{N}$ and note that by the monotone convergence theorem it suffices to show that
	\[
	\|\sup_{n\le n_0}A_{2^n}|f|\|_{\ell^p(\mathbb{Z}^2)}\lesssim_{p,c} \|f\|_{\ell^p(\mathbb{Z}^2)}
	\]
	where the implied constant is independent of the truncation parameter $n_0$. We note that 
	\begin{multline}
	\Big\|\sup_{n\le n_0}\big|\mathcal{F}^{-1}_{\mathbb{R}^2}\big[\eta\mathcal{L}_{2^n}\mathcal{F}_{\mathbb{R}^2}f\big]\big|\Big\|_{L^p(\mathbb{R}^2)}\le\Big\|\sup_{n\le n_0}\big|\mathcal{F}^{-1}_{\mathbb{R}^2}\big[\mathcal{L}_{2^n}\mathcal{F}_{\mathbb{R}^2}f\big]\big|\Big\|_{L^p(\mathbb{R}^2)}+\Big\|\sup_{n\le n_0}\big|\mathcal{F}^{-1}_{\mathbb{R}^2}\big[(1-\eta)\mathcal{L}_{2^n}\mathcal{F}_{\mathbb{R}^2}f\big]\big|\Big\|_{L^p(\mathbb{R}^2)}
	\\
	\le\big\|\sup_{t\in \mathbb{N}}|B_tf|\big\|_{L^p(\mathbb{R}^2)}+\Big(\sum_{n\in\mathbb{N}}\big\|\mathcal{F}^{-1}_{\mathbb{R}^2}\big[(1-\eta)\mathcal{L}_{2^n}\mathcal{F}_{\mathbb{R}^2}f\big]\big\|_{L^p(\mathbb{R}^2)}^p\Big)^{1/p}\text{.}
	\end{multline}
	The first summand is immediately bounded by the previous theorem, and for the second summand, by Proposition~$\ref{multapproxforpetclean}$  we have that
	\[
	\big\|\mathcal{F}^{-1}_{\mathbb{R}^2}\big[(1-\eta)\mathcal{L}_{2^n}\mathcal{F}_{\mathbb{R}^2}f\big]\big\|_{L^2(\mathbb{R}^2)}\le\|(1-\eta)\mathcal{L}_{2^n}\|_{L^{\infty}}\|f\|_{L^2(\mathbb{R}^2)}\lesssim 2^{-\frac{\tau}{4}n}\|f\|_{L^2(\mathbb{R}^2)}\text{,}
	\]
	because for large $n$ we have that $\supp(1-\eta)\subseteq\mathbb{R}^2\setminus\widetilde{\mathfrak{M}}_{2^n}$. This immediately implies that 
	\[
	\Big\|\sup_{n\le n_0}\big|\mathcal{F}^{-1}_{\mathbb{R}^2}\big[\eta\mathcal{L}_{2^n}\mathcal{F}_{\mathbb{R}^2}f\big]\big|\Big\|_{L^2(\mathbb{R}^2)}\lesssim \|f\|_{L^2(\mathbb{R}^2)}\text{.}
	\]
	
	If $p\neq 2$ we may fix $p_0\in(1,\infty)$ such that $p$ is strictly between $p_0$ and $2$. We have that
	\begin{multline}
	\big\|\mathcal{F}^{-1}_{\mathbb{R}^2}\big[(1-\eta)\mathcal{L}_{2^n}\mathcal{F}_{\mathbb{R}^2}f\big]\big\|_{L^{p_0}(\mathbb{R}^2)}\le
	\big\|\mathcal{F}^{-1}_{\mathbb{R}^2}\big[\mathcal{L}_{2^n}\mathcal{F}_{\mathbb{R}^2}f\big]\big\|_{L^{p_0}(\mathbb{R}^2)}+\big\|\mathcal{F}^{-1}_{\mathbb{R}^2}\big[\eta\mathcal{L}_{2^n}\mathcal{F}_{\mathbb{R}^2}f\big]\big\|_{L^{p_0}(\mathbb{R}^2)}
\\
	=\big\|B_{2^n}f\big\|_{L^{p_0}(\mathbb{R}^2)}+\big\|\mathcal{F}^{-1}_{\mathbb{R}^2}[\eta]*\mathcal{F}^{-1}_{\mathbb{R}^2}\big[\mathcal{L}_{2^n}\mathcal{F}_{\mathbb{R}^2}f\big]\big\|_{L^{p_0}(\mathbb{R}^2)}
	\\
	\le\|f\|_{L^{p_0}(\mathbb{R}^2)}+\|\mathcal{F}^{-1}_{\mathbb{R}^2}\eta\|_{L^1(\mathbb{R})}\big\|B_{2^n}f\big\|_{L^{p_0}(\mathbb{R}^2)}\lesssim_{\eta}\|f\|_{L^{p_0}(\mathbb{R}^2)}\text{.}
	\end{multline} 
	Riesz–Thorin interpolation between $p_0$ and $2$ yields a bound of the form
	\[
	\big\|\mathcal{F}^{-1}_{\mathbb{R}^2}\big[(1-\eta)\mathcal{L}_{2^n}\mathcal{F}_{\mathbb{R}^2}f\big]\big\|_{L^p(\mathbb{R}^2)}\lesssim_{\eta} 2^{-\frac{\tau(p,p_0)}{4}n}\|f\|_{L^p(\mathbb{R}^2)}\text{,}\quad\text{for some $\tau(p,p_0)>0$.}
	\]
	 Thus we obtain
	\[
	\Big\|\sup_{n\le n_0}\big|\mathcal{F}^{-1}_{\mathbb{R}^2}\big[\eta\mathcal{L}_{2^n}\mathcal{F}_{\mathbb{R}^2}f\big]\big|\Big\|_{L^p(\mathbb{R}^2)}\lesssim
	\|f\|_{L^p(\mathbb{R}^2)}+\Big(\sum_{n\in\mathbb{N}}2^{-\frac{p\tau(p,p_0)}{4}n}\|f\|^p_{L^p(\mathbb{R}^2)}\Big)^{1/p}\lesssim \|f\|_{L^p(\mathbb{R}^2)}\text{.}
	\]
	We may use the previous bound together with an application of Proposition~2.1 from \cite{TransferenceLp} to conclude that
	\begin{equation}\label{easyhelp}
		\Big\|\sup_{n\le n_0}\big|\mathcal{F}^{-1}_{\mathbb{Z}^2}\big[\eta\mathcal{L}_{2^n}\mathcal{F}_{\mathbb{Z}^2}f\big]\big|\Big\|_{\ell^p(\mathbb{Z}^2)}\lesssim \|f\|_{\ell^p(\mathbb{Z}^2)}\text{,}
	\end{equation}
	where the implied constant is independent of $n_0$. Finally, note that
	\begin{equation}\label{almostdonemax}
		\|\sup_{n\le n_0}A_{2^n}|f|\|_{\ell^p(\mathbb{Z}^2)}\le\big\|\sup_{n\le n_0}\big|\mathcal{F}^{-1}_{\mathbb{Z}^2}\big[\eta\mathcal{L}_{2^n}\mathcal{F}_{\mathbb{Z}^2}f\big]\big|\big\|_{\ell^p(\mathbb{Z}^2)}+\Big(\sum_{n\in\mathbb{N}_0}\big\|\mathcal{F}^{-1}_{\mathbb{Z}^2}\big[(k_{2^n}-\eta\mathcal{L}_{2^n})\mathcal{F}_{\mathbb{Z}^2}f\big]\big\|^p_{\ell^p(\mathbb{Z}^2)}\Big)^{1/p}
	\end{equation}
	and the first summand is immediately appropriately bounded while for the second summand an argument analogous to the one given earlier yields the desired estimate. More precisely, for $p=2$ note that 
	\[
	\big\|\mathcal{F}^{-1}_{\mathbb{Z}^2}\big[(k_{2^n}-\eta\mathcal{L}_{2^n})\mathcal{F}_{\mathbb{Z}^2}f\big]\big\|_{\ell^2(\mathbb{Z}^2)}\le\|k_{2^n}-\eta\mathcal{L}_{2^n}\|_{L^{\infty}([-1/2,1/2))}\|f\|_{\ell^2(\mathbb{Z})}\lesssim 2^{-\frac{\tau}{4}n}\|f\|_{\ell^2(\mathbb{Z})}\text{,}
	\] 
	where for the last estimate we have used Proposition~$\ref{multapproxforpetclean}$ and the fact that for large $n$, $\eta(x)=1$ for all $x\in\widetilde{\mathfrak{M}}_{2^n}$ and thus 
	\begin{equation}\label{reasonforclose}
		\|k_{2^n}-\eta\mathcal{L}_{2^n}\|_{L^{\infty}([-1/2,1/2))}\le \|k_{2^n}-\mathcal{L}_{2^n}\|_{L^{\infty}(\widetilde{\mathfrak{M}}_{2^n})}+\big\||k_{2^n}|+|\mathcal{L}_{2^n}|\big\|_{L^{\infty}([-1/2,1/2)\setminus \widetilde{\mathfrak{M}}_{2^n})}\lesssim 2^{-\frac{\tau}{4}n}\text{.}
	\end{equation}
	This implies that \[
	\|\sup_{n\le n_0}A_{2^n}|f|\|_{\ell^2(\mathbb{Z}^2)}\le \|f\|_{\ell^2(\mathbb{Z}^2)}\text{.}
	\]
	If $p\neq 2$ we again fix $p_0$ such $p$ is strictly between $p_0$ and $2$ and we interpolate between the good $\ell^2$-bounds and the trivial $\ell^{p_0}$-bounds by noting
	\[
	\big\|\mathcal{F}^{-1}_{\mathbb{Z}^2}\big[(k_{2^n}-\eta\mathcal{L}_{2^n})\mathcal{F}_{\mathbb{Z}^2}f\big]\big\|_{\ell^{p_0}(\mathbb{Z}^2)}\le \big\|A_{2^n}f\big\|_{\ell^{p_0}(\mathbb{Z}^2)}+\big\|\mathcal{F}^{-1}_{\mathbb{Z}^2}\big[(\eta\mathcal{L}_{2^n})\mathcal{F}_{\mathbb{Z}^2}f\big]\big\|_{\ell^{p_0}(\mathbb{Z}^2)}\lesssim \|f\|_{\ell^{p_0}(\mathbb{Z}^2)}\text{,} 
	\] 
	where we may use for example \eqref{easyhelp}. Interpolation gives an estimate
	\[
	\big\|\mathcal{F}^{-1}_{\mathbb{Z}^2}\big[(k_{2^n}-\eta\mathcal{L}_{2^n})\mathcal{F}_{\mathbb{Z}^2}f\big]\big\|_{\ell^{p}(\mathbb{Z}^2)}\lesssim 2^{-\tau(p,p_0)n}\|f\|_{\ell^p(\mathbb{Z}^2)}\text{,}
	\]
	which is view of \eqref{almostdonemax} gives $\|\sup_{n\le n_0}A_{2^n}|f|\|_{\ell^p(\mathbb{Z}^2)}\lesssim \|f\|_{\ell^p(\mathbb{Z}^2)}$ with implied constant independent of $n_0$ and thus we have established the following estimate
	\[
	\big\|\sup_{t\in\mathbb{N}}|A_tf|\big\|_{\ell^p(\mathbb{Z}^2)}\le C(c,p)\|f\|_{\ell^p(\mathbb{Z}^2)}
	\text{.}
	\]
	Now we fix $\lambda\in(1,\infty)$, $r\in(2,\infty)$ and focus on the $r$-variations. Note that by Theorem~\ref{contcounterpart} we will have
	\[
	\big\|V^r\big(\mathcal{F}^{-1}_{\mathbb{R}^2}\big[\mathcal{L}_{\lfloor\lambda^n\rfloor}\mathcal{F}_{\mathbb{R}^2}f\big]:\,n\in\mathbb{N}_0\big)\big\|_{L^2(\mathbb{R}^2)}=\big\|V^r\big(B_{\lfloor \lambda^n\rfloor}f:\,n\in\mathbb{N}_0\big)\big\|_{L^2(\mathbb{R}^2)}\lesssim_r\|f\|_{L^2(\mathbb{R}^2)}\text{.}
	\] 
	Let $\tilde{\eta}\in\mathcal{C}^{\infty}_c(\mathbb{R}^2)$ such that $\cic{1}_{[-2^{-7},2^{-7}]}\le \tilde{\eta}\le \cic{1}_{[-2^{-6},2^{-6}]}$, and note that according to Proposition~3.1 in \cite{MVV} (applied for $Q=1$ and $m=(0,0)$) we have that
	\begin{multline}
	\big\|V^r\big(\mathcal{F}^{-1}_{\mathbb{Z}^2}\big[\mathcal{L}_{\lfloor\lambda^n\rfloor}\tilde{\eta}\mathcal{F}_{\mathbb{Z}^2}f\big]:\,n\in\mathbb{N}_0\big)\big\|_{\ell^2(\mathbb{Z}^2)}
\\
\lesssim_r \big\|\mathcal{F}^{-1}_{\mathbb{Z}^2}\big[\tilde{\eta}\mathcal{F}_{\mathbb{Z}^2}f\big]\big\|_{\ell^2(\mathbb{Z}^2)}=\big\|\tilde{\eta}\mathcal{F}_{\mathbb{Z}^2}f\big\|_{L^2(\mathbb{T}^2)}\le\|\mathcal{F}_{\mathbb{Z}^2}f\|_{L^2(\mathbb{T}^2)}=\|f\|_{\ell^2(\mathbb{Z}^2)}\text{.}
	\end{multline}
	This immediately implies the result since $r>2$ and we have
	\begin{multline}
	\|V^r(A_{\lfloor \lambda^n\rfloor}f:\,n\in\mathbb{N}_0)\|_{\ell^2(\mathbb{Z})}\le \big\|V^r\big(\mathcal{F}^{-1}_{\mathbb{Z}^2}\big[\mathcal{L}_{\lfloor\lambda^n\rfloor}\tilde{\eta}\mathcal{F}_{\mathbb{Z}^2}f\big]:\,n\in\mathbb{N}_0\big)\big\|_{\ell^2(\mathbb{Z}^2)}
	\\
	+
	\Big\|\Big(\sum_{n\in\mathbb{N}_0}\big|\mathcal{F}^{-1}_{\mathbb{Z}^2}\big[k_{\lfloor\lambda^n\rfloor}\mathcal{F}_{\mathbb{Z}^2}f\big]-\mathcal{F}^{-1}_{\mathbb{Z}^2}\big[\mathcal{L}_{\lfloor\lambda^n\rfloor}\tilde{\eta}\mathcal{F}_{\mathbb{Z}^2}f\big]\big|\Big)^{1/2}\Big\|_{\ell^2(\mathbb{Z}^2)}\lesssim_r \|f\|_{\ell^2(\mathbb{Z}^2)}
\\	+\bigg(\sum_{n\in\mathbb{N}_0}\big\|\mathcal{F}^{-1}_{\mathbb{Z}^2}\big[(k_{\lfloor\lambda^n\rfloor}-\mathcal{L}_{\lfloor\lambda^n\rfloor}\tilde{\eta})\mathcal{F}_{\mathbb{Z}^2}f\big]\big\|_{\ell^2(\mathbb{Z}^2)}^2\bigg)^{1/2}\\
\lesssim
	\|f\|_{\ell^2(\mathbb{Z}^2)}+	\Big(\sum_{n\in\mathbb{N}_0}\big\|k_{\lfloor\lambda^n\rfloor}-\mathcal{L}_{\lfloor\lambda^n\rfloor}\tilde{\eta}\big\|_{L^\infty(\mathbb{T}^2)}^2\Big)^{1/2}\|f\|_{\ell^2(\mathbb{Z}^2)}\lesssim \|f\|_{\ell^2(\mathbb{Z}^2)}\text{,}
\end{multline}
	where for the last estimate we used the fact that for large $n$ we have that $\big\|k_{\lfloor\lambda^n\rfloor}-\mathcal{L}_{\lfloor\lambda^n\rfloor}\tilde{\eta}\big\|_{L^\infty(\mathbb{T}^2)}\lesssim \lambda^{-\frac{\tau}{4}n}$, and this can be justified in an analogous manner to \eqref{reasonforclose}. The proof is complete.
\end{proof}
\section{Reducing the proof of Theorem~$\ref{maintheorem}$ to the continuous case}
In this section we begin the proof of Theorem~$\ref{maintheorem}$, essentially reducing it to Theorem~1.1 in \cite{GuoOscillatory}. We will use the following Sobolev embedding lemma which can be found in \cite[Lemma 4.12]{benbook}; we state it here for the sake of clarity.\begin{lemma} \label{sobolev}
	Assume $(X,\mathcal{B},\mu)$ is a $\sigma$-finite measure space and $I\subseteq \mathbb{R}$ an interval. Assume $\big(F(x,t)\big)_{x\in X,t\in I}$ is a family of functions such that $F(x,\cdot)$ is absolutely continuous for $\mu$-a.e. $x\in X$. Then we have
	\[
	\Big\| \sup_{t \in I}|F(x,t)| \Big\|_{L^2_{\mu}(X)}\lesssim A+\left(Aa|I|\right)^{\frac{1}{2}}\text{,}
	\]
	where $A= \sup_{t \in I}\| F(x,t)\|_{L^2_{\mu}(X)}$ and $a= \sup_{t \in I}\| \partial_t F(x,t)\|_{L^2_{\mu}(X)}$, and the implied constant is absolute.
\end{lemma}
Before appealing to this result, we discard all $\lambda$-frequencies in the supremum defining the $c$-modulated Carleson operators; this is the content of the following proposition, the proof of which will be given in the end of the section and is the most technical part of the present work. Similarly to the beginning of Section~3 we fix for convenience two parameters $\varepsilon\coloneqq\frac{2-c}{4}\in\big(0,\min\big\{\frac{1}{4},2-c\big\}\big)$ and $\nu\coloneqq \frac{\varepsilon}{8}$  and define for every $j\in\mathbb{N}$ 
\[
X_j\coloneqq \big\{\lambda\in[-1/2,1/2):\,|\lambda|\le 2^{(-c+\nu)j}\big\}\text{.}
\]
With this in mind the following proposition holds.
\begin{proposition}\label{awayfrom0}There exist an absolute positive constant $C$ such that for every $j\in\mathbb{N}$ and $\mathsf{i}\in\{0,1\}$ we have
\begin{equation}\label{discardlest}
\Big\|\sup_{\lambda\in[-1/2,1/2)\setminus X_j}\big|\big(\hat{f}(\cdot)m_j^{\mathsf{i}}(\cdot,\lambda)\big)^{\vee}(x)\big|
\Big\|_{\ell^2_{\d x}}\le C(c-1)^{-\frac{1}{2}}(2-c)^{-\frac{9}{4}} 2^{-\frac{\varepsilon}{256} j}\|f\|_{\ell^2}
\text{.}
\end{equation}
\end{proposition} 
Assuming momentarily the validity of the previous proposition we obtain for $\mathsf{i}\in\{0,1\}$
\begin{multline}\label{firstinterpolation}
	\bigg\|\sup_{\lambda \in [-1/2,1/2)}\Big| \sum_{n \neq 0}f(x-n) \frac{e(\lambda \mathsf{sign}^{\mathsf{i}}(n) \lfloor |n|^{c} \rfloor)}{n}\Big|\bigg\|_{{\ell_{\d x}^p}}
	\\
	\le\bigg\|\sup_{\lambda\in[-1/2,1/2)}\Big|\Big(\sum_{j\in\mathbb{N}}\hat{f}(\cdot)m_j^{\mathsf{i}}(\cdot,\lambda)\cic{1}_{X_j}(\lambda)\Big)^{\vee}(x)\Big|\bigg\|_{\ell_{\d x}^p}+\sum_{j\in\mathbb{N}}\bigg\|\sup_{\lambda\in[-1/2,1/2)\setminus X_j}\Big|\big(\hat{f}(\cdot)m_j^{\mathsf{i}}(\cdot,\lambda)\big)^{\vee}(x)\Big|\bigg\|_{\ell_{\d x}^p}
	\\
	\lesssim_{c,p}\bigg\|\sup_{\lambda\in[-1/2,1/2)}\Big|\Big(\sum_{j\in\mathbb{N}}\hat{f}(\cdot)m_j^{\mathsf{i}}(\cdot,\lambda)\cic{1}_{X_j}(\lambda)\Big)^{\vee}(x)\Big|\bigg\|_{\ell_{\d x}^p}+\|f\|_{\ell^p}\text{,}
\end{multline}
where to obtain the last inequality we argue as follows. For $p=2$, \eqref{discardlest} yields that 
\begin{multline}\label{absoluteconstants}
\sum_{j\in\mathbb{N}}\Big\|\sup_{\lambda\in[-1/2,1/2)\setminus X_j}\big|\big(\hat{f}(\cdot)m_j^{\mathsf{i}}(\cdot,\lambda)\big)^{\vee}(x)\big|
\Big\|_{\ell^2_{\d x}}
\\
\lesssim (c-1)^{-\frac{1}{2}}(2-c)^{-\frac{9}{4}}\Big(\sum_{j\in\mathbb{N}} 2^{-\frac{2-c}{256} j}\Big)\|f\|_{\ell^2}\lesssim (c-1)^{-\frac{1}{2}}(2-c)^{-\frac{13}{4}}\|f\|_{\ell^2}\text{,}
\end{multline}
where the implicit constant is absolute and we used that for all $\theta\in(0,1)$ we have
\begin{equation}\label{sumwithdecaynow}
\sum_{j\in\mathbb{N}} 2^{-\theta j}\le\frac{1}{1-2^{-\theta }}=\frac{1}{1-e^{-\theta\log(2) }}\le \frac{e}{\theta\log(2)}\lesssim \theta^{-1}\text{, since $1-e^{-x}\ge \frac{x}{e}$ for $x\in[0,1]$.
}
\end{equation}
 We note that the following pointwise estimate holds
\begin{equation}\label{maxHLd}
	\sup_{\lambda\in[-1/2,1/2)}\big|\big(\hat{f}(\cdot)m_j^{\mathsf{i}}(\cdot,\lambda)\big)^{\vee}(x)\big|\le\sum_{n\in\mathbb{Z}}|f(x-n)\psi_j(n)|\lesssim Mf(x)\text{,}
\end{equation}
where $Mf$ is the standard discrete Hardy--Littlewood maximal operator. We obtain weak-type $(1,1)$ and strong $(\infty,\infty)$ estimates and interpolation with the $\ell^2$-estimates of \eqref{discardlest} yields
\begin{equation}\label{afterinterpolation}
	\bigg\|\sup_{\lambda\in[-1/2,1/2)\setminus X_j}\Big|\big(\hat{f}(\cdot)m_j^{\mathsf{i}}(\cdot,\lambda)\big)^{\vee}(x)\Big|\bigg\|_{\ell_{\d x}^p}\lesssim_{c,p} 2^{-\rho_p j}\|f\|_{\ell^p}\text{, for some $\rho_p>0$,}
\end{equation}
and the last inequality in \eqref{firstinterpolation} is justified.

We are left with the task of estimating the first term in the last line of \eqref{firstinterpolation}.  We will use Proposition~$\ref{approxcom}$ together with a similar reasoning. Note that
\begin{multline}\label{inputapprox}
	\bigg\|\sup_{\lambda\in[-1/2,1/2)}\Big|\Big(\sum_{j\in\mathbb{N}}\hat{f}(\cdot)m_j^{\mathsf{i}}(\cdot,\lambda)\cic{1}_{X_j}(\lambda)\Big)^{\vee}(x)\Big|\bigg\|_{\ell_{\d x}^p}\\
	\le \bigg\|\sup_{\lambda\in [-1/2,1/2)}\Big|\Big(\sum_{j\in\mathbb{N}}\hat{f}(\cdot)H_j^{\mathsf{i}}(\cdot,\lambda)\cic{1}_{X_j}(\lambda)\Big)^{\vee}(x)\Big|\bigg\|_{\ell_{\d x}^p}+\sum_{j\in\mathbb{N}}\bigg\|\sup_{\lambda\in X_j}\big|\big(\hat{f}(\cdot)E_j^{\mathsf{i}}(\cdot,\lambda)\big)^{\vee}(x)\big|\bigg\|_{\ell_{\d x}^p}\text{,}
\end{multline}
where the definition of $H_j^{\mathsf{i}}$ and $E_j^{\mathsf{i}}$ can be found in Proposition~$\ref{approxcom}$. By Lemma~$\ref{sobolev}$ we get that
\begin{equation}
\frac{\bigg\|\sup_{\lambda\in X_j}\big|\big(\hat{f}(\cdot)E_j^{\mathsf{i}}(\cdot,\lambda)\big)^{\vee}(x)\big|\bigg\|_{\ell_{dx}^2}}{(c-1)^{-\frac{1}{2}}}\lesssim \big(2^{-\frac{\varepsilon}{4} j}+(2^{-\frac{\varepsilon}{4} j}2^{cj}2^{(-c+\nu)j})^{1/2}\big)\|f\|_{\ell^2}\lesssim 2^{-\frac{\varepsilon}{16}j}\|f\|_{\ell^2}\text{,}
\end{equation}
since $\nu=\varepsilon/8$. For $p=2$ we can immediately use the above estimate to bound the second summand in the last line of \eqref{inputapprox} as follows
\begin{equation}\label{est512}
\sum_{j\in\mathbb{N}}\bigg\|\sup_{\lambda\in X_j}\big|\big(\hat{f}(\cdot)E_j^{\mathsf{i}}(\cdot,\lambda)\big)^{\vee}(x)\big|\bigg\|_{\ell_{dx}^2}\lesssim (c-1)^{-\frac{1}{2}}\Big(\sum_{j=0}^{\infty}2^{-\frac{\varepsilon}{16}j}\Big)\|f\|_{\ell^2}\lesssim (2-c)^{-1}(c-1)^{-\frac{1}{2}}\|f\|_{\ell^2}\text{,}
\end{equation}
where we took into account that $\varepsilon=(2-c)/4$ and used a calculation as in \eqref{sumwithdecaynow}.

For $p\neq 2$ we  use interpolation between the $\ell^2$-estimates and trivial estimates to obtain an estimate of the form 
\begin{equation}\label{noreapeats}
\bigg\|\sup_{\lambda\in X_j}\big|\big(\hat{f}(\cdot)E_j^{\mathsf{i}}(\cdot,\lambda)\big)^{\vee}(x)\big|\bigg\|_{\ell_{dx}^{p}}\lesssim_{c,p} 2^{-\rho_p j}\|f\|_{\ell^p}\text{, for some $\rho_p>0$,}
\end{equation}
This can be achieved by taking into account that $E_j^{\mathsf{i}}=m_j^{\mathsf{i}}-H_j^{\mathsf{i}}$, the estimate \eqref{maxHLd} and by establishing
\[
\bigg\|\sup_{\lambda\in X_j}\big|\big(\hat{f}(\cdot)H_j^{\mathsf{i}}(\cdot,\lambda)\big)^{\vee}(x)\big|\bigg\|_{\ell_{dx}^{p_0}}\lesssim_{c,p_0}\|f\|_{\ell^{p_0}}\text{,} 
\]
which can be easily done by proving the corresponding one for $L^{p_0}$ with an argument as in \eqref{maxHLd} and then using the sampling principle in Proposition~2.1 in \cite{TransferenceLp} as in \cite{carlesonprimes}, see page 1052.

The estimate \eqref{noreapeats} immediately appropriately bounds the second summand in the last line of \eqref{inputapprox} and all that is left is to bound the first summand i.e. to establish the estimate
\[
\bigg\|\sup_{\lambda\in [-1/2,1/2)}\Big|\Big(\sum_{j\in\mathbb{N}}\hat{f}(\cdot)H_j^{\mathsf{i}}(\cdot,\lambda)1_{X_j}(\lambda)\Big)^{\vee}(x)\Big|\bigg\|_{\ell_{\d x}^p}\lesssim_{c,p} \|f\|_{\ell^p}\text{.}
\]
Again, appealing to Proposition~2.1 in \cite{TransferenceLp} as in \cite{carlesonprimes}, it suffices to establish its continuous analogue, namely, the following
\begin{equation}
	\bigg\|\sup_{\lambda\in [-1/2,1/2)}\Big|\mathcal{F}^{-1}_{\mathbb{R}}\Big[\sum_{j\in\mathbb{N}}\mathcal{F}_{\mathbb{R}}[f](\cdot)H_j^{\mathsf{i}}(\cdot,\lambda)1_{X_j}(\lambda)\Big]\Big|\bigg\|_{L^p(\mathbb{R})}\lesssim_{c,p} \|f\|_{L^p(\mathbb{R})}\text{,}
\end{equation}
which is the content of Proposition~$\ref{ContRedExpl}$. All complications associated with the discrete nature of our problem have been addressed.

Moreover, it is clear from our proof $\big(\eqref{firstinterpolation},\eqref{absoluteconstants},\eqref{inputapprox},\eqref{est512}\big)$,  that for the case of $p=2$ we have shown that
\begin{multline}\label{Sec5absconstant}
\sup_{f\in\ell^2:\,\|f\|_{\ell^2}\le 1}\bigg\|\sup_{\lambda \in [-1/2,1/2)}\Big| \sum_{n \neq 0}f(x-n) \frac{e(\lambda \mathsf{sign}^{\mathsf{i}}(n) \lfloor |n|^{c} \rfloor)}{n}\Big|\bigg\|_{{\ell_{\d x}^2}}
\lesssim (c-1)^{-\frac{1}{2}}(2-c)^{-\frac{13}{4}}
\\
+\sup_{f\in L^2(\mathbb{R}):\,\|f\|_{L^2(\mathbb{R})}\le 1} \bigg\|\sup_{\lambda\in [-1/2,1/2)}\Big|\mathcal{F}^{-1}_{\mathbb{R}}\Big[\sum_{j\in\mathbb{N}}\mathcal{F}_{\mathbb{R}}[f](\cdot)H_j^{\mathsf{i}}(\cdot,\lambda)1_{X_j}(\lambda)\Big]\Big|\bigg\|_{L^2(\mathbb{R})}\text{,}
\end{multline}
and unsuprisingly the constant from the continuous counterpart corresponding to the second summand above does not introduce any blowup as $c\to 2^-$ and is in fact $\lesssim( c-1)^{-\frac{1}{2}}$ as will be apparent from the proof of Proposition~$\ref{ContRedExpl}$ is Section~6.
We end this section by proving Proposition~$\ref{awayfrom0}$. 
\begin{proof}[Proof of Proposition~$\ref{awayfrom0}$.]
	As in the beginning of Lemma~$\ref{continuousvdc}$ it suffices to prove the estimate for $j\in\mathbb{N}$ such that $(2-c) j\ge 2^{20}$, since for the rest it is immediate. We remove the floor function. For $n\neq 0$, $\lambda\in[-1/2,1/2)$ and $M\in\mathbb{N}$ we have
	\[
\begin{split}
		e(\lambda \mathsf{sign}^{\mathsf{i}}(n) \lfloor |n|^c \rfloor )&=e(\lambda \mathsf{sign}^{\mathsf{i}}(n)|n|^c) e(-\lambda \mathsf{sign}^{\mathsf{i}}(n) \left\{ |n|^c\right\} )\\ &= \sum_{|m| \leq M}c_m(\lambda \mathsf{sign}^{\mathsf{i}}(n))e\big((m+\lambda \mathsf{sign}^{\mathsf{i}}(n))|n|^c\big)+O \left( \min\left\{ 1,\frac{1}{M  \big\| |n|^c \big\|}\right\} \right)\text{,}
\end{split}
	\] 
	as in the beginning of the proof of Proposition~$\ref{expsummin}$. We apply the previous estimate with $M=\lfloor 2^{\varepsilon j}\rfloor$ and obtain
	\begin{multline}\label{splitinsigns}
		\sup_{\lambda \in X_j^c}\big|\big(\hat{f}(\cdot)m_j^{ \mathsf{i} }(\cdot,\lambda)\big)^{\vee}(x)\big|=\sup_{\lambda\in X_j^c}\Big|\sum_{n\neq 0}f(x-n)e(\lambda \mathsf{sign}^{\mathsf{i}}(n) \lfloor |n|^c\rfloor)\psi_j(n)\Big|
		\\
		\lesssim
		\sup_{\lambda\in X_j^c}\Big|\sum_{n\neq 0}\sum_{|m|\le M}c_m(\lambda \mathsf{sign}^{\mathsf{i}}(n))f(x-n)e((m+\lambda \mathsf{sign}^{\mathsf{i}}(n))|n|^c)\psi_j(n)\Big|
		\\
		+\sum_{n\neq 0}|f(x-n)\psi_j(n)|\min\bigg\{ 1,\frac{1}{M  \big\| |n|^c \big\|}\bigg\}
		\\	\lesssim\sup_{\lambda\in X_j^c }  \left| \sum_{n \neq 0}\sum_{|m| \leq M}c_{m}(\lambda \mathsf{sign}^{\mathsf{i}}(n))f(x-n) e((\lambda \mathsf{sign}^{\mathsf{i}}(n)+m) |n|^c) \psi_j(n) \right|+ |f|*L_j(x)\text{,}
	\end{multline}
	where 
	\[
	L_j(x)=2^{-j}1_{2^{j-3}\le |x|\le 2^{j-1}}\min\bigg\{1,\frac{1}{M\big\||x|^c\big\|}\bigg\}\text{,}
	\]
and notice that by taking into account our choice of $M$, the estimate \eqref{mindyadic} and since $\varepsilon j=\frac{2-c}{4}j\ge1$ and  $\log(M)\le \log(2^{\varepsilon j}+1)\le \log(2^{\varepsilon j+1})\le \log(2^{2\varepsilon j})\lesssim (2-c)j$,  we get
\begin{multline}
	\|L_j\|_{\ell^1}\lesssim 2^{-j}\sum_{2^{j-3}\le n\le 2^{j-1}}\min\bigg\{1,\frac{1}{M\|n^c\|}\bigg\}\lesssim (c-1)^{-\frac{1}{2}}2^{-j}j\varepsilon \big(2^{(1-\varepsilon )j}+2^{(c/2+\varepsilon /2)j}\big)
\\
\lesssim (c-1)^{-\frac{1}{2}}\Big(j\varepsilon2^{-\varepsilon j}+j\varepsilon 2^{\frac{c+\varepsilon-2}{2}j}\Big)\lesssim(c-1)^{-\frac{1}{2}}j\varepsilon 2^{-\varepsilon j}\lesssim(c-1)^{-\frac{1}{2}}2^{-\frac{\varepsilon}{2} j}\text{,}
\end{multline}
since $c+\varepsilon-2<-2\varepsilon$, because $\varepsilon<\frac{2-c}{3}$, and by \eqref{cleantrick}. Thus $\|L_j\|_{\ell^1}\lesssim (c-1)^{-\frac{1}{2}}2^{-\frac{\varepsilon}{2} j}$. Now by defining $T_jf\coloneqq f*L_j$, we see that 
\[
\|T_j\|_{\ell^2\to\ell^2}\lesssim (c-1)^{-\frac{1}{2}}2^{-\frac{\varepsilon}{2} j}\text{.}
\]
For $\tau\in\{+,-\}$, let 
	\[
	\psi_j^{\tau}=\left\{
\begin{array}{ll}
      \psi_j \cic{1}_{(0,\infty) }\text{,} & \text{if }\tau=+ \\
      \psi_j \cic{1}_{(-\infty,0) }\text{,} &\text{if }\tau=-\text{.}\\
\end{array} 
\right.
	\]
Note that the first summand in the last line of \eqref{splitinsigns} is bounded by
	\begin{multline}
	\sum_{\mathsf{\tau} \in \left\{+,-\right\}}	\sup_{\lambda\in X_j^c } \sum_{|m| \leq M}\frac{1}{|m|+1} \left| \sum_{n \neq 0}f(x-n) e((\lambda \mathsf{sign}^{\mathsf{i}}(n)+m) |n|^c) \psi_j^{\tau}(n) \right|
		\\
		\lesssim \sum_{\mathsf{\tau} \in \left\{+,-\right\}} \sup_{2^{(-c+\nu)j}<|\lambda|\le M+1/2}\left|\sum_{n \neq 0}f(x-n) e(\lambda |n|^c) \psi_j^{\tau}(n) \right|\log(M+1)
		\\
\lesssim (2-c)j \max_{\mathsf{\tau} \in \left\{+,-\right\}} \sup_{2^{(-c+\nu)j}<|\lambda|\le 2^{2\varepsilon j}}\Big|\sum_{n \neq 0}f(x-n) e(\lambda |n|^c) \psi_j^{\tau}(n) \Big|\text{,}
	\end{multline}
since $\varepsilon j=\frac{2-c}{4}j\ge1$ and  $\log(M+1)\le \log(2^{\varepsilon j}+1)\le \log(2^{\varepsilon j+1})\le \log(2^{2\varepsilon j})\lesssim (2-c)j$. Combining everything yields
	\begin{multline}\label{firstboundin5.2}
	\Big\|\sup_{\lambda\in X^c_j}\big|\big(\hat{f}(\cdot)m_j^{\mathsf{i}}(\cdot,\lambda)\big)^{\vee}(x)\big|
	\Big\|_{\ell^2_{\d x}}\lesssim (2-c) j  \max_{{\mathsf{\tau} \in \left\{+,-\right\}}} \bigg\|\sup_{2^{(-c+\nu)j}<|\lambda|\le 2^{2\varepsilon j}}\Big|\sum_{n \neq 0}f(x-n) e(\lambda  |n|^c) \psi_j^{\tau}(n) \Big|\bigg\|_{\ell^2_{\d x}}
\\
+(c-1)^{-\frac{1}{2}}2^{-\frac{\varepsilon}{2} j}\|f\|_{\ell^2}\text{,}
	\end{multline}
	and we focus on the first term of the right hand side. We claim that it suffices to prove that for each $\tau \in  \left\{+,-\right\}$ we have
\begin{equation}\label{est51goal}
\bigg\|\sup_{2^{\nu j}<|\lambda2^{cj}|\le 2^{(c+2\varepsilon)j}}\Big|\sum_{n \neq 0}f(x-n) e(\lambda |n|^c) \psi_j^{\tau}(n) \Big|\bigg\|_{\ell^2_{\d x}}\lesssim (2-c)^{-1}j \big((2-c)(c-1)\big)^{-\frac{1}{4}}2^{-\frac{\varepsilon}{128}j}\|f\|_{\ell^2}\text{,}
\end{equation}
To see why this is the case note that once we establish \eqref{est51goal} we may return to \eqref{firstboundin5.2} to obtain
\begin{multline}\label{firstboundin5.21}
	\Big\|\sup_{\lambda\in X^c_j}\big|\big(\hat{f}(\cdot)m_j^{\mathsf{i}}(\cdot,\lambda)\big)^{\vee}(x)\big|
	\Big\|_{\ell^2_{\d x}}\lesssim  j^2  \big((2-c)(c-1)\big)^{-\frac{1}{4}}2^{-\frac{\varepsilon}{128}j}\|f\|_{\ell^2}
+(c-1)^{-\frac{1}{2}}2^{-\frac{\varepsilon}{2} j}\|f\|_{\ell^2}
\\
\lesssim  (2-c)^{-\frac{9}{4}}(c-1)^{-\frac{1}{4}}(\varepsilon j)^22^{-\frac{\varepsilon}{128}j}\|f\|_{\ell^2}
+(c-1)^{-\frac{1}{2}}2^{-\frac{\varepsilon}{2} j}\|f\|_{\ell^2}
\lesssim (2-c)^{-\frac{9}{4}}(c-1)^{-\frac{1}{2}}2^{-\frac{\varepsilon}{256}j}\|f\|_{\ell^2}
\text{,}
\end{multline}
where we took into account the fact that $\varepsilon=\frac{2-c}{4}$ and used the estimate $x^22^{-x}\lesssim 2^{-x/2}$.

To establish \eqref{est51goal} it will be important to have further control of the scale of $2^{cj}\lambda$, so we perform a further decomposition. Note that
\begin{multline}\label{refinewithr}
\sup_{2^{\nu j}<\lambda2^{cj}\le 2^{(c+2\varepsilon )j}}\Big|\sum_{n \neq 0}f(x-n) e(\lambda |n|^c) \psi_j^{\tau}(n) \Big|
\\
\le\sum_{r\in(2-c)\mathbb{Z}}\sup_{2^{\nu j}<\lambda2^{cj}\le 2^{(c+2\varepsilon)j}}\cic{1}_{[2^{r},2^{r+(2-c)})}(\lambda 2^{cj})\Big|\sum_{n \neq 0}f(x-n) e(\lambda |n|^c) \psi_j^{\tau}(n) \Big|
\\	\le\sum_{\substack{r\in(2-c)\mathbb{Z}\text{,}\\ \nu j-2\le r\le (c+2\varepsilon)j+2}}\sup_{2^r\le \lambda 2^{cj}<2^{r+(2-c)}}\Big|\sum_{n \neq 0}f(x-n) e(\lambda |n|^c) \psi_j^{\tau}(n) \Big|\text{,}
	\end{multline}
	where for the last equality we simply note that for any $r\in(2-c)\mathbb{Z}\footnote{Here we use the standard notation $\lambda\mathbb{Z}\coloneqq \{\lambda k:\,k\in\mathbb{Z}\}$. Letting $r\in\mathbb{Z}$ instead of $r\in(2-c)\mathbb{Z}$ results in an exponential loss in the constants and since we try to optimize the constant the argument yields with respect to $c$, a finer control of the scale $2^{c j}\lambda$ is needed.}$ such that $r<\nu j-2$ or $r>(c+2\varepsilon)j+2$ we have that
	\[
	\big[2^r,2^{r+(2-c)}\big)\cap\big(2^{\nu j},2^{(c+2\varepsilon)j}\big]=\emptyset\text{.}
	\]
	Note that the outer sum has $O((2-c)^{-1}j)$ terms and thus to establish \eqref{est51goal} it suffices to prove that for any $r\in[\nu j-2, (c+2\varepsilon)j+2]$ we have that
\begin{equation}\label{estimate51}
	\bigg\|\sup_{2^r\le \lambda 2^{cj}<2^{r+1}}\Big|\sum_{n \neq 0}f(x-n) e(\lambda |n|^c) \psi_j^{\tau}(n) \Big|\bigg\|_{\ell^2_{\d x}}\lesssim  \big((2-c)(c-1)\big)^{-\frac{1}{4}}2^{-\frac{\varepsilon}{128}j}\|f\|_{\ell^2}\text{,}
\end{equation}
 We linearize the supremum; we will prove that for any measurable $\lambda_x\colon\mathbb{Z}\to [2^{r-cj},2^{r+(2-c)-cj})$ we have
	\begin{equation}\label{goalreturn}
		\Big\|\sum_{n \neq 0}f(x-n) e(\lambda_x |n|^c) \psi_j^{\tau}(n) \Big\|_{\ell^2_{\d x}}\lesssim \big((2-c)(c-1)\big)^{-\frac{1}{4}}2^{-\frac{\varepsilon}{128}j}\|f\|_{\ell^2}\text{.}
	\end{equation}
We fix such a function and use the $TT^*$ method here. It is not difficult to see that if
\[
T_{\tau,j}f(x)=	\sum_{n \neq 0}f(x-n) e(\lambda_x |n|^c) \psi_j^{\tau}(n)\text{,}\quad \text{then}\quad T^*_{\tau,j}g(y)=\sum_{x\in\mathbb{Z}}g(x)e(-\lambda_x|x-y|^c)\psi_j^{\tau}(x-y)\text{,}
	\]
	and also that
	\[
	T_{\tau,j}T^*_{\mathsf{\tau},j}f(x)=\sum_{y\in\mathbb{Z}}f(y)K_{\tau,j}(x,y)\text{,}
	\]
	where
	\[
	K_{\tau,j}(x,y)=\sum_{m\in\mathbb{Z}}e(\lambda_x|x-m|^c-\lambda_y |y-m|^c)\psi_j^{\tau}(x-m)\psi_j^{\tau}(y-m)\text{.}
	\]
	Let us remark that $K_{\tau,j}$ depends on $r$ as well, since $\lambda$ depends on $r$, we chose to suppress that dependence. Since $\|T_{\tau,j}\|_{\ell^2\to\ell^2}= \|T_{\tau,j}T^*_{\tau,j}\|_{\ell^2\to\ell^2}^{1/2}$, we focus on the operator $
	T_{\tau,j}T^*_{\tau,j}$, and to derive the desired bounds it suffices to establish the following proposition.
	\begin{proposition}\label{keyptwiseest}
Fix $\rho\coloneqq\frac{\varepsilon}{64}=\frac{2-c}{256}$ and $\sigma\coloneqq  \frac{\varepsilon}{16}=\frac{2-c}{64}$. There exists an absolute positive constant $C$ such that the following estimate holds
\begin{equation}\label{TTstarest}
|K_{\tau,j}(x,y)|\lesssim\big((2-c)(c-1)\big)^{-\frac{1}{2}} 2^{-(1+\rho)j}\cic{1}_{2^{(1-\sigma)j}\le |x-y|\le 2^j}+2^{-j}\cic{1}_{|x-y|\le 2^{(1-\sigma)j}}\text{.}
\end{equation}
\end{proposition}
\begin{proof}We begin with certain straightforward reductions. Firstly, we assume that $y\ge x$ and $\tau=-$ since symmetric arguments yield the result for the remaining cases. Secondly, we may assume that $|x-y|\geq 2^{(1-\sigma)j}$ since otherwise we clearly have
		\[
		|K_{-,j}(x,y)|\lesssim \sum_{m:\,2^{j-3}\le|x-m|\le 2^{j-1} }2^{-2j}\le2^{-j}\text{.}
		\]
		Now for the case $y\ge x$ and $|x-y| \geq 2^{(1-\sigma)j}$ we note that
		\begin{multline}\label{expsum2}
			K_{-,j}(x,y)
			\\			=\bigg(\sum_{ \substack{x \leq  m \leq y \\  2^{j-3}\le|m-y|\le2^{j-1} \\    2^{j-3}\le|m-x|\le2^{j-1}}}+ \sum_{\substack{m>y \\ 2^{j-3}\le|m-y|\le2^{j-1} \\    2^{j-3}\le|m-x|\le2^{j-1}}}+\sum_{\substack{m< x \\ 2^{j-3}\le|m-y|\le2^{j-1} \\    2^{j-3}\le|m-x|\le2^{j-1}}} \bigg) e(\lambda_x|x-m|^c-\lambda_y |y-m|^c) \psi^{-}_j(x-m) \psi^{-}_j(y-m)
			\\
			=\sum_{\substack{m>y \\ 2^{j-3}\le|m-y|\le2^{j-1} \\    2^{j-3}\le|m-x|\le2^{j-1}}}e(\lambda_x|x-m|^c-\lambda_y |y-m|^c) \psi^{-}_j(x-m) \psi^{-}_j(y-m)\text{,}
		\end{multline}
		since for the first sum we have that $x-m$ and $y-m$ have different signs and for the third one we note that $x-m>0$ and thus $\psi_j^{-}(x-m)=0$. We focus on bounding the last line of $\eqref{expsum2}$. For any fixed $x,y$ the phase of the exponential sum becomes 
\[
f(t)=\lambda_x(t-x)^c-\lambda_y(t-y)^c\text{,}\quad t\in[a_{x,y},b_{x,y}]\text{,}
\]
where $a_{x,y},b_{x,y}$ are defined by $\{m\in\mathbb{Z}:\,m>y,\,2^{j-3}\le|m-y|\le 2^{j-1},\,2^{j-3}\le|m-x|\le 2^{j-1}\}\eqqcolon\{a_{x,y},\dotsc,b_{x,y}\}$. We see that 
		\[
		f''(t)=c(c-1)\big(\lambda_x(t-x)^{c-2}-\lambda_y(t-y)^{c-2}\big)=c(c-1)\Big(\big(\lambda_x^{\frac{1}{c-2}}(t-x)\big)^{c-2}-\big(\lambda_y^{\frac{1}{c-2}}(t-y)\big)^{c-2}\Big)\text{.}
		\]
		By the Mean Value Theorem we obtain that there exists $\xi=\xi(x,y,t)$ between
		$\lambda_x^{\frac{1}{c-2} }(t-x)$ and $ \lambda_y^{\frac{1}{c-2} }(t-y)$ such that
\begin{multline}
		|f''(t)|=|c(c-1)(c-2)|\xi^{c-3}\Big|\lambda_x^{\frac{1}{c-2} }(t-x)-\lambda_y^{\frac{1}{c-2} }(t-y)\Big|
		\\
		=|c(c-1)(c-2)|\lambda_x^{\frac{1}{c-2} }\xi^{c-3}\bigg|(t-x)-\frac{\lambda_y^{\frac{1}{c-2} }}{\lambda_x^{\frac{1}{c-2} }}(t-y)\bigg|\text{,}
	\end{multline}
		and now define $\varepsilon(x,y)\coloneqq \big(\frac{\lambda_y}{\lambda_x}\big)^{\frac{1}{c-2}}-1$. Without loss of generality let us assume that $\lambda_y\le \lambda_x$ and thus $\varepsilon(x,y)\ge 0$. For the case $\lambda_y\ge \lambda_x$, note that  
		\[
		|f''(t)|=|c(c-1)(c-2)|\lambda_y^{\frac{1}{c-2} }\xi^{c-3}\bigg|(t-y)-\frac{\lambda_x^{\frac{1}{c-2} }}{\lambda_y^{\frac{1}{c-2} }}(t-x)\bigg|\text{,}
		\]
		and one could define $\tilde{\varepsilon}(x,y)\coloneqq \big(\frac{\lambda_x}{\lambda_y}\big)^{\frac{1}{c-2}}-1\ge 0$ and proceed in an identical manner. Thus from now one we assume $\lambda_y\le \lambda_x$ and we may write
\begin{multline}
|f''(t)|=|c(c-1)(c-2)|\lambda_x^{\frac{1}{c-2} }\xi^{c-3}\big|(t-x)-\big(1+\varepsilon(x,y)\big)(t-y)\big|
\\
=|c(c-1)(c-2)|\lambda_x^{\frac{1}{c-2} }\xi^{c-3}\big|(y-x)-\varepsilon(x,y)(t-y)\big|\text{.}
\end{multline}
Now we distinguish two cases. Firstly, we assume that $\varepsilon(x,y)\le2^{-\sigma j-100}$. Note that then
\[
\lambda_x^{\frac{1}{c-2}}\big(\lambda_x^{\frac{1}{c-2}}(t-x)\big)^{c-3}\simeq \lambda_x^{\frac{1}{c-2}} \lambda_x^{\frac{c-3}{c-2}}2^{j(c-3)}\simeq \big(2^{cj}\lambda_x\big)2^{-3j} \simeq 2^{r-3j}
		\]
		and the same reasoning together with the fact that $\lambda_y^{\frac{1}{c-2}}\simeq \lambda_x^{\frac{1}{c-2}}$ yield	
\[		\lambda_x^{\frac{1}{c-2}}\big(\lambda_y^{\frac{1}{c-2}}(t-x)\big)^{c-3}\simeq \lambda_y^{\frac{1}{c-2}}\big(\lambda_y^{\frac{1}{c-2}}(t-x)\big)^{c-3}\simeq 2^{r-3j}\text{.}
\]
Thus $\lambda_x^{\frac{1}{c-2}}\xi^{c-3}\simeq 2^{r-3j}$, and this yields
\[
2^{r-3j}\big(2^{(1-\sigma)j}-2^{-\sigma j-100}2^{j-1}\big)\le 2^{r-3j}\big(|y-x|-2^{-\sigma j-100}|t-y|\big)\lesssim\frac{|f''(t)|}{(c-1)(2-c)}\lesssim 2^{r-3j}2^j=2^{r-2j}\text{,}
		\] 
		thus $2^{r-2j-\sigma j}\lesssim\frac{|f''(t)|}{(c-1)(2-c)}\lesssim 2^{r-2j}$. By van der Corput, see \cite{IK} page 208, for any interval $J\subseteq \{a_{x,y},\dotsc,b_{x,y}\}$ we get 
		\begin{multline}
		\bigg|\sum_{m\in J}e\left( \lambda_x   |x-m|^c -\lambda_y   |y-m|^c \right)\bigg|\lesssim 2^{\sigma j}\big((c-1)(2-c)\big)^{\frac{1}{2}}2^{\frac{r}{2}-j-\frac{\sigma}{2}j}2^j+\big((c-1)(2-c)\big)^{-\frac{1}{2}}2^{-\frac{r}{2}+j+\frac{\sigma}{2}j}
\\
\lesssim2^{\frac{\sigma}{2}j}\Big(2^{\frac{r}{2}}+\big((c-1)(2-c)\big)^{-\frac{1}{2}}2^{j-\frac{r}{2}}\Big)\text{.}
\end{multline}
By Lemma~5.5 from \cite{Demeter} we conclude
\[
|K_{-,j}(x,y)|\lesssim 2^{-2j+\frac{\sigma}{2}j+\frac{r}{2}}+\big((c-1)(2-c)\big)^{-1/2}2^{-j+\frac{\sigma}{2}j-\frac{r}{2}}\lesssim\big((c-1)(2-c)\big)^{-1/2}2^{-j}\big(2^{-j+\frac{\sigma}{2}j+\frac{r}{2}}+2^{\frac{\sigma}{2}j-\frac{r}{2}}\big)\text{,}
\]
and note that since $r\in\big[\nu j-2,(c+2\varepsilon)j+2\big]$, we get
\[
-j+\frac{\sigma}{2}j+\frac{r}{2}\le -j+\frac{\sigma}{2}j+\frac{1}{2}(c+2\varepsilon)j+1=\frac{c-2+\sigma+2\varepsilon}{2}j+1
\]
and
\[
\frac{\sigma}{2}j-\frac{r}{2}\le \frac{\sigma}{2}j-\frac{\nu }{2}j+1\text{.}
\]
Remembering that $\sigma= \frac{\varepsilon}{16}$, $\nu= \frac{\varepsilon}{8}$ and $\varepsilon=\frac{2-c}{4}$, it is easy to check that the previous estimate yields that
\begin{equation}\label{firstestwithnumerology}
|K_{-,j}(x,y)|\lesssim\big((c-1)(2-c)\big)^{-1/2} 2^{-(1+\rho)j}\text{,}\quad\text{for $\rho=\frac{\varepsilon}{64}$\text{.}}
\end{equation}

It remains to treat the case when $\varepsilon(x,y)\ge2^{-\sigma j-100}$. Fix $\nu'\coloneqq\frac{ 3\sigma}{2}=\frac{3\varepsilon}{32}$ and define
\[
I_{x,y}\coloneqq \{m\in\{a_{x,y}\dotsc,b_{x,y}\}:\,|(y-x)-\varepsilon(x,y)(m-y)|\le 2^{j-\nu'j}\}\text{, and thus }|I_{x,y}|\lesssim 2^{(1+\sigma-\nu')j}\text{.} 
\]
Let $T_{x,y}\coloneqq \{a_{x,y},\dotsc,b_{x,y}\}$ and split it to $T_{x,y}=I_{x,y}\cup(T_{x,y}\setminus I_{x,y})$. Note that $I_{x,y}$ is an interval and thus $T_{x,y}\setminus I_{x,y}$ is a union of at most two intervals i.e. there exists two intervals $I^1_{x,y}$ and $I^2_{x,y}$ such that $T_{x,y}\setminus I_{x,y}=I^1_{x,y}\cup I^2_{x,y}$. We split our exponential sum to $I_{x,y}\cup I^1_{x,y}\cup I^2_{x,y}$ and we bound these seperately.

For $I_{x,y}$ note that
\begin{equation}
\Big|\sum_{m\in I_{x,y}}e\left( \lambda_x   |x-m|^c -\lambda_y  |y-m|^c  \right) \psi^-_j(x-m) \psi_j^-(y-m)\Big|\lesssim 2^{-2j}|I_{x,y}|\lesssim 2^{-(1+\nu'-\sigma)j}\lesssim2^{-(1+\frac{\varepsilon}{64})j}\text{,}
\end{equation}
since $\nu'-\sigma=\frac{\sigma}{2}=\frac{\varepsilon}{32}$.

We now treat $I^1_{x,y}$, and $I^2_{x,y}$ can be handled likewise. Let $c_{x,y}$ and $d_{x,y}$ be the left and right endpoints of $I^1_{x,y}$, and note that for any $t\in[c_{x,y},d_{x,y}]$ we have 
\[
2^{r-3j}2^{j-\nu'j}\lesssim\frac{ |f''(t)|}{(c-1)(2-c)}\lesssim 2^{r-3j}2^j\text{,}
\] 
where for the first estimate we used the fact that if $t\in[c_{x,y},d_{x,y}]$, then $|(y-x)-\varepsilon(x,y)(t-y)|>2^{j-\nu'j}$, and for the second estimate we used the fact that $\varepsilon(x,y)\lesssim 1$, since 
\[
\varepsilon(x,y)\coloneqq \Big(\frac{\lambda_y}{\lambda_x}\Big)^{\frac{1}{c-2}}-1\le \Big(\frac{\lambda_x}{\lambda_y}\Big)^{\frac{1}{2-c}}\le\bigg(\frac{2^{r+(2-c)-cj}}{2^{r-cj}}\bigg)^{\frac{1}{2-c}}\le 2\text{.}\footnote{Notice that had we opted for the natural choice $r\in\mathbb{Z}$ for our spliting in \eqref{refinewithr} to control the scale of $2^{cj}\lambda$, here we would have $\varepsilon(x,y)\lesssim 2^{\frac{1}{2-c}}$ resulting in exponential blowup of the final estimate as $c\to 2^{-}$.}
\]
We use van der Corput to conclude that for any subinterval $(I^1_{x,y})'$ of $I^1_{x,y}$ we get
\begin{multline}
\Big|\sum_{m\in (I^1_{x,y})'}e\left( \lambda_x  |x-m|^c -\lambda_y  |y-m|^c  \right) \Big|
\\
\lesssim \big((2-c)(c-1)\big)^{\frac{1}{2}}2^{\nu' j}2^{\frac{r}{2}-j-\frac{\nu'}{2}j}2^j+\big((2-c)(c-1)\big)^{-\frac{1}{2}}2^{-\frac{r}{2}+j+\frac{\nu'}{2}j}
\lesssim 
2^{\frac{\nu'}{2}j+\frac{r}{2}}+\big((2-c)(c-1)\big)^{-\frac{1}{2}}2^{-\frac{r}{2}+j+\frac{\nu'}{2}j}\text{,}
\end{multline}
and thus similarly to our previous considerations
\begin{multline}
\Big|\sum_{m\in I^1_{x,y}}e\left( \lambda_x  |x-m|^c -\lambda_y  |y-m|^c  \right)\psi^-_j(x-m)\psi_j^-(y-m) \Big|
\\
\lesssim 2^{\frac{\nu'}{2}j+\frac{r}{2}-2j}+\big((2-c)(c-1)\big)^{-\frac{1}{2}}2^{-\frac{r}{2}-j+\frac{\nu'}{2}j}\lesssim\big((2-c)(c-1)\big)^{-\frac{1}{2}}2^{-j}\Big(2^{\frac{\nu'}{2}j+\frac{r}{2}-j}+2^{-\frac{r}{2}+\frac{\nu'}{2}j}\Big)
\\
\lesssim\big((2-c)(c-1)\big)^{-\frac{1}{2}}2^{-j(1+\frac{\varepsilon}{64})}
\text{,}
\end{multline}
where for the last estimate we took into account that $r\in\big[\nu j-2,(c+2\varepsilon)j+2\big]$ as well as the fact that $\nu'=\frac{3\varepsilon}{32}$, $\nu=\frac{\varepsilon}{8}$, $\varepsilon\coloneqq\frac{2-c}{4}$ and thus
\[
\frac{\nu'}{2}j+\frac{r}{2}-j<\frac{\nu'}{2}j-j+\frac{c+2\varepsilon}{2}j+1=\frac{\nu'+c-2+2\varepsilon}{2}j+1=\frac{\frac{3\varepsilon}{32}-4\varepsilon+2\varepsilon}{2}j+1\le -\frac{\varepsilon}{64}j+1
\]
and
\[
-\frac{r}{2}+\frac{\nu'}{2}j\le \frac{\nu'}{2}j-\frac{\nu}{2} j+1=\frac{\frac{3\varepsilon}{32}-\frac{4\varepsilon}{32}}{2} j+1 \le -\frac{\varepsilon}{64}j+1\text{.}
\]
The estimate is complete and the same reasoning applies for $I^2_{x,y}$.
\end{proof}
	Using the estimate of Proposition~$\ref{keyptwiseest}$ we can immediately establish the estimate \eqref{goalreturn} and conclude the proof of Proposition~$\ref{awayfrom0}$ since \eqref{TTstarest} implies that
	\begin{equation}\label{kerneltobound}
	|T_{\tau,j}T_{\tau,j}^*f(x)|\lesssim \big((2-c)(c-1)\big)^{-\frac{1}{2}}(2^{-\rho j}+2^{-\sigma j})Mf(x)\lesssim\big((2-c)(c-1)\big)^{-\frac{1}{2}}2^{-\rho j}Mf(x)\text{,}
	\end{equation}
	and thus we get $\|T_{\tau,j}\|_{\ell^2\to\ell^2}\lesssim \big((2-c)(c-1)\big)^{-\frac{1}{4}}2^{-\frac{\rho}{2}j}$, establishing \eqref{estimate51} and concluding the proof.
\end{proof}

\section{Concluding the Proof of Theorem~$\ref{maintheorem}$}
This section is devoted to establishing the estimate \eqref{contreduction} which we do in the following proposition, concluding the proof of Theorem~$\ref{maintheorem}$. \begin{proposition}\label{ContRedExpl}Let $\mathsf{i}\in\{0,1\}$, $p \in (1,\infty)$. There exists a positive constant $C=C(p,c)$ such that for all $f\in L^p(\mathbb{R})$ we have
	\begin{equation}\label{contreduction}
		\bigg\|\sup_{\lambda\in [-1/2,1/2)}\Big|\mathcal{F}^{-1}_{\mathbb{R}}\Big[\sum_{j\in\mathbb{N}}\mathcal{F}_{\mathbb{R}}[f](\cdot)H_j^{\mathsf{i}}(\cdot,\lambda)\cic{1}_{X_j}(\lambda)\Big]\Big|\bigg\|_{L^p(\mathbb{R})}\le C\|f\|_{L^p(\mathbb{R})}\text{.}
	\end{equation}
\end{proposition}
The objective in this section is to reduce the proposition above to Theorem~1.1 from \cite{GuoOscillatory} which we state here for the sake of clarity.
\begin{theorem}\label{Guo}For every $\varepsilon\in\mathbb{R}$ and $
\mathsf{i}\in\{0,1\}$, let  
\begin{equation}\label{GuoOperator}
\mathcal{C}^{\mathsf{i}}_{\varepsilon}f(x)=\sup_{\lambda \in \mathbb{R}} \bigg|\int_{\mathbb{R}}f(x-t) \frac{e(\lambda [t]_{\mathsf{i}}^{\varepsilon})}{t}\d t\bigg|=\sup_{\lambda \in \mathbb{R}} \Big| \mathcal{F}_{\mathbb{R}}^{-1} \Big[ \mathcal{F}_{\mathbb{R}}[f](\xi) \; \mathsf{p.v.} \int_{\mathbb{R}} \frac{e(\lambda [t]_{\mathsf{i}}^{\varepsilon}-\xi t)}{t}\d t \Big](x)\Big|  \text{.}
\end{equation}
Then for every $\varepsilon\neq 1-\mathsf{i}$ and $p \in (1,\infty)$ there exists a  constant $C_{p,\varepsilon,\mathsf{i}}$ such that $\|\mathcal{C}^{\mathsf{i}}_{\varepsilon}f\|_{L^p(\mathbb{R})}\le C_{p,\varepsilon,\mathsf{i}} \|f\|_{L^p(\mathbb{R})}$.
\end{theorem}

\begin{proof}The proof can be found \cite{GuoOscillatory}.
\end{proof}
We briefly mention that for $\varepsilon\in(1,2)$ we have $C_{2,\varepsilon,\mathsf{i}}\lesssim (c-1)^{-\frac{1}{2}}$, where the implied constant is absolute. This is shown in Lemma~$\ref{GuoConstant}$.

We are ready to give the proof of Proposition~$\ref{ContRedExpl}$.
\begin{proof}[Proof of Proposition~$\ref{ContRedExpl}$] We decompose as follows  
	\[
	\sum_{j \in \N} H_j^{\mathsf{i}}(\xi,\lambda) \cic{1}_{X_j}(\lambda)=\sum_{j \in \Z}H_j^{\mathsf{i}}(\xi,\lambda)- \sum_{j \leq 0} H_j^{\mathsf{i}}(\xi,\lambda) -\sum_{j \in \N} H_j^{\mathsf{i}}(\xi,\lambda)  \cic{1}_{X_j^c}
	\]
	and
	\[ \sum_{j \in \N} H_j^{\mathsf{i}}(\xi ,\lambda) \cic{1}_{X_j^c}(\lambda) =\sum_{ \ell \in \Z} \sum_{j \in \N} \cic{1}_{[1,2)}(\lambda^{\frac{1}{c}  }2^{j-\ell})H_j^{\mathsf{i}}(\xi,\lambda) \cic{1}_{X_j^{c}}(\lambda)= \sum_{\ell  \geq \frac{\nu}{c}-1}H^{\mathsf{i}}_{j(\lambda,\ell)}(\xi,\lambda)\cic{1}_{\mathbb{N}}(j(\lambda,\ell)) \cic{1}_{X_{j(\lambda,\ell)}^c}(\lambda)\text{,}
	\]
	where $j(\lambda,\ell)$ is the unique integer such that $1 \leq \lambda^{\frac{1}{c}}2^{j(\lambda,\ell)-\ell}<2$. Therefore we may write
	\begin{multline}\label{decompositionH}
		\bigg\|\sup_{\lambda \in [-1/2,1/2)} \Big| \mathcal{F}^{-1}_{\mathbb{R}} \Big[\sum_{j \in \N} \mathcal{F}_{\mathbb{R}}[f](\cdot) H_j^{\mathsf{i}}(\cdot,\lambda) \cic{1}_{X_j}(\lambda) \Big] \Big|\bigg\|_{L^p(\mathbb{R})}
		\\
		\le \bigg\| \sup_{ \lambda \in [-1/2,1/2)} \Big|\mathcal{F}^{-1}_{\mathbb{R}} \Big[\mathcal{F}_{\mathbb{R}}[f](\cdot) \sum_{j \in \Z}H_j^{\mathsf{i}}(\cdot,\lambda) \Big] \Big|\bigg\|_{L^p(\mathbb{R})}
		\\
		+\bigg\|\sup_{ \lambda \in [-1/2,1/2)} \Big|\mathcal{F}^{-1}_{\mathbb{R}}\Big[\sum_{\ell\ge\frac{\nu}{c}-1}\mathcal{F}_{\mathbb{R}}[f](\cdot) \cic{1}_{[1,2)}(\lambda^{\frac{1}{c}  }2^{j-\ell})  H^{\mathsf{i}}_{j(\lambda,\ell)}(\cdot,\lambda)) \cic{1}_{X_{j(\lambda,\ell)}^{c}}(\lambda)\Big]\Big|\bigg\|_{L^p(\mathbb{R})}
		\\
		+\bigg\|\sup_{\lambda \in [-1/2,1/2)}\Big|\mathcal{F}_{\mathbb{R}}^{-1}\Big[\sum_{j \leq 0} \mathcal{F}_{\mathbb{R}}[f](\cdot) H_j^{\mathsf{i}}(\cdot,\lambda) \Big]\Big|\bigg\|_{L^p(\mathbb{R})}\eqqcolon S_1+S_2+S_3\text{,}
	\end{multline}
	and the task is reduced to bounding these three summands. The first one will be bounded using Theorem~$\ref{Guo}$, the second one using a simple variant of the proof of Proposition~$\ref{awayfrom0}$ and the third one will be treated via pointwise bounds by the Hardy-Littlewood maximal function and the maximal truncations of the Hilbert transform.\\
	\textbf{Estimates for $S_1$.} Note that
	\[
	\begin{split}
		&\sum_{j \in \mathbb{Z}}H^{\mathsf{i}}_j(\xi,\lambda) \mathcal{F}_{\mathbb{R}}[f](\xi)= \eta_1(\lambda) \big(\eta_1(\xi)\mathcal{F}_{\mathbb{R}}[f](\xi)\big)\sum_{j \in \mathbb{Z}} \int_{\mathbb{R}} e(\lambda [t]_{\mathsf{i}}^c-\xi t) \psi_j(t) \d t
		\\
		&=\eta_1(\lambda) \mathcal{F}_{\mathbb{R}}\big[f* \mathcal{F}_{\mathbb{R}}^{-1}[\eta_1]\big](\xi) \bigg(\mathsf{p.v.} \int_{\mathbb{R}} e(\lambda [t]_{\mathsf{i}}^c-\xi t) \frac{\d t}{t}\bigg)\text{,}
	\end{split}
	\]
	and thus for every $\lambda\in[-1/2,1/2)$ and $x\in\mathbb{R}$, since $\|\eta_1\|_{L^{\infty}}\le 1$, we get
	\begin{multline}\label{Cevend}
		\Big|\mathcal{F}^{-1}_{\mathbb{R}} \Big[\mathcal{F}_{\mathbb{R}}[f](\cdot) \sum_{j \in \Z}H^{\mathsf{i}}_j(\cdot,\lambda) \Big](x)\Big|\le \bigg|\mathcal{F}^{-1}_{\mathbb{R}}\bigg[\mathcal{F}_{\mathbb{R}}\big[f* \mathcal{F}_{\mathbb{R}}^{-1}[\eta_1]\big](\xi) \bigg(\mathsf{p.v.} \int_{\mathbb{R}} e(\lambda [t]_{\mathsf{i}}^c-\xi t) \frac{\d t}{t}\bigg)\bigg](x)\bigg|
		\\
		\le
		\mathcal{C}_{c}^{\mathsf{i}}[f* \mathcal{F}_{\mathbb{R}}^{-1}[\eta_1]](x)\text{,}
	\end{multline}
	and finally we obtain 
	\begin{equation}\label{ptwiseboundconv}
		S_1\le \big\|\mathcal{C}_{c}^{\mathsf{i}}\big[f* \mathcal{F}_{\mathbb{R}}^{-1}[\eta_1]\big]\big\|_{L^p(\mathbb{R})}\lesssim_{c,p} \|f* \mathcal{F}_{\mathbb{R}}^{-1}[\eta_1]\|_{L^p(\mathbb{R})}\lesssim_{\eta_1}\|f\|_{L^p(\mathbb{R})}\text{,}
	\end{equation}
	where we have used Theorem~$\ref{Guo}$. The estimate for $S_1$ is complete, and we note that for the case $p=2$, we obtain $S_1\lesssim (c-1)^{-\frac{1}{2}}\|f\|_{L^2(\mathbb{R})}$, see our comment after Theorem~$\ref{Guo}$ and Lemma~$\ref{GuoConstant}$.
\\ \, \\
	\textbf{Estimates for $S_2$. }The estimation relies on the following lemma.
	\begin{lemma} \label{largeelllp} Let $\mathsf{i}\in\{0,1\}$. For every $p\in(1,\infty)$ there exist positive constants $C=C(c,p)$ and $\beta=\beta(c,p)$ such that for every $\ell\in\mathbb{N}$ we have
		\[
		\Big\|\sup_{j\in\mathbb{Z}} \sup_{|\lambda|^{\frac{1}{c}}2^{j}\simeq 2^{\ell}} \Big| \mathcal{F}_{\mathbb{R}}^{-1}\big[\mathcal{F}_{\mathbb{R}}(f)(\cdot) H_j^{\mathsf{i}}(\cdot,\lambda)\big]\Big|\Big\|_{L^p(\mathbb{R})}  \le C 2^{-\beta \ell}\|f\|_{L^p(\mathbb{R})}\text{.}
		\]
Moreover, for $p=2$, one may choose $C\lesssim (c-1)^{-\frac{1}{2}}$ and $\beta=\frac{1}{6}$.
	\end{lemma}
	Before giving the proof let us see how this immediately yields the desired estimate for $S_2$. Note that
	\begin{multline}
		S_2\le \sum_{\ell\ge\frac{\nu}{c}-1}\bigg\|\sup_{ \lambda \in [-1/2,1/2)} \Big|\mathcal{F}^{-1}_{\mathbb{R}}\Big[\mathcal{F}_{\mathbb{R}}[f](\cdot) \cic{1}_{[1,2)}(\lambda^{\frac{1}{c}  }2^{j-\ell})  H^{\mathsf{i}}_{j(\lambda,\ell)}(\cdot,\lambda)) \cic{1}_{X_{j(\lambda,\ell)}^{c}}(\lambda)\Big]\Big|\bigg\|_{L^p(\mathbb{R})}
		\\
		\le\sum_{\ell\ge\frac{\nu}{c}-1}\Big\|\sup_{j\in\mathbb{Z}} \sup_{|\lambda|^{\frac{1}{c}}2^{j}\simeq 2^{\ell}} \Big| \mathcal{F}_{\mathbb{R}}^{-1}\big[\mathcal{F}_{\mathbb{R}}(f)(\cdot) H_j^{\mathsf{i}}(\cdot,\lambda)\big]\Big|\Big\|_{L^p(\mathbb{R})}
		\\\lesssim_{c,p} \sum_{\ell\ge\frac{\nu}{c}-1}2^{-\beta \ell}\|f\|_{L^p(\mathbb{R})}\lesssim_p \|f\|_{L^p(\mathbb{R})}\text{,}
	\end{multline}
	and the estimate for $S_2$ is complete. For the case $p=2$ we additionally note that by taking into account the second assertion of Lemma~$\ref{largeelllp}$ we get
\begin{equation}\label{EstS2L2}
S_2\lesssim (c-1)^{-\frac{1}{2}}\Big(\sum_{\ell\ge\frac{\nu}{c}-1}2^{-\frac{\ell}{6} }\Big)\|f\|_{L^2(\mathbb{R})}\le  (c-1)^{-\frac{1}{2}}\Big(\sum_{\ell\ge-1}2^{-\frac{\ell}{6} }\Big)\|f\|_{L^2(\mathbb{R})}\lesssim  (c-1)^{-\frac{1}{2}}\|f\|_{L^2(\mathbb{R})}\text{.}
\end{equation}
We now provide a proof of the above lemma.
	\begin{proof}[Proof of Lemma~$\ref{largeelllp}$]
		With an argument analogous to the one given in Section~5, one sees that it suffices to establish the second assertion of the lemma and specifically we prove that the $L^2$ norm of the following operator
\[
f \mapsto \int_{\R} f(x-t)e(\lambda_x[t]_{\mathsf{i}}^c) \psi^{\tau}_{j_x}(t) \d t
\] 
is controlled by $\lesssim (c-1)^{-\frac{1}{2}}2^{-\frac{\ell}{6}}$ for any choice of measurable functions $\lambda_{x}\colon\mathbb{Z}\to[0,\infty)$ and $j_x\colon\mathbb{Z}\to\mathbb{Z}$ with the property that $\lambda_x^{\frac{1}{c}} 2^{j_x} \simeq 2^{\ell} $, where the implied constants are absolute. The corresponding result on $L^p$ follows from interpolation with identical considerations with the ones of Section~5.
To that end we prove a pointwise bound on the kernel of the associated $TT^*$ operator defined by
\begin{equation}\label{kerneldefinition}
K_{\mathsf{i},\tau,\ell}(x,y)= \int_{\R} e(\lambda_x[x-t]_{\mathsf{i}}^c-\lambda_y[y-t]_{\mathsf{i}}^c) \psi^{\tau}_{j_x}(x-t) \psi^{\tau}_{j_y}(y-t) \d t\text{.}
\end{equation}
More precisely, we prove that for $\sigma\coloneqq \frac{1}{3}$ we have the following pointwise bound
\begin{multline}\label{ttkernelbound}
\left|K_{\mathsf{i},\tau,\ell}(x,y)\right| 
\\
\lesssim (c-1)^{-1}2^{-\max \left\{j_x,j_y\right\}-\sigma\ell} \cic{1}_{ 2^{\max \left\{j_x,j_y\right\}-\sigma\ell  } \lesssim |x-y| \lesssim  2^{\max \left\{j_x,j_y\right\}}}+2^{-\max \left\{j_x,j_y\right\}} \cic{1}_{|x-y| \lesssim 2^{\max \left\{j_x,j_y\right\}-\sigma\ell}} 
\text{.} \end{multline}
This establishes the desired $L^2$-bound with an argument as in \eqref{kerneltobound}.
		We have that
\[
\left|K_{\mathsf{i},\tau,\ell}(x,y)\right| \lesssim   2^{-\max \left\{j_x,j_y\right\}} \cic{1}_{|x-y| \lesssim 2^{\max \left\{j_x,j_y\right\}}}\text{,} \] 
and thus if $|x-y| \lesssim 2^{\max \left\{j_x,j_y\right\}-\sigma\ell }$ we have that \eqref{ttkernelbound} is satisfied. Therefore, we are going to assume that $|x-y| \gtrsim 2^{\max \left\{j_x,j_y\right\}-\sigma\ell }$. Without loss of generality let $x \leq y.$ We are also going to denote by $g$ the phase of the oscillatory integral we are interested in. Let us mention that all $t\in[x,y]$ may be discarded in the integration in \eqref{kerneldefinition} since $x-t$, $y-t$ will have different signs and the integrand will be zero. We begin with the case $|j_x-j_y| \geq 100$. We have \[\begin{split}
			& \partial_{t}g(t)=  \eps_{t,x} \lambda_xc|x-t|^{c-1}- \eps_{t,y} \lambda_{y}c |y-t|^{c-1}, \quad \text{with } |\eps_{t,x}|=|\eps_{t,y}|=1 \\ & \Rightarrow |\partial_{t}g| \gtrsim 2^{c\ell} 2^{- \min \left\{j_x,j_y\right\}}
			\end{split}  \]  Since there are $O(1)$  subintervals of $I_{x,y} \coloneqq \left\{ |x-t| \sim 2^{j_x}, |y-t| \sim 2^{j_y} \right\}$ on each of  which $\partial_tg$ is monotonous we may apply van der Corput's lemma to obtain \[ |K(x,y)| \lesssim  \frac{2^{-j_x-j_y}}{2^{c\ell-\min \left\{j_x,j_y\right\}}} \cic{1}_{|x-y| \lesssim 2^{\max \left\{j_x,j_y\right\}}}=2^{-c\ell-\max \left\{j_x,j_y\right\}} \cic{1}_{|x-y| \lesssim 2^{\max \left\{j_x,j_y\right\}}}. \] Therefore, from now on we may focus on the interesting case when the scales are dyadically the ``same'', namely when $|j_x-j_y| \leq 100$. We will repeatedly use the fact that $2^{j_x} \sim 2^{j_y}$. Without loss of generality we assume that $\lambda_x \leq \lambda_y$ and let us define
\[
\eps(x,y)=\Big(\frac{\lambda_y}{\lambda_x}\Big)^{\frac{1}{c-1}}-1\text{.}
\]
If $\eps(x,y) \leq 2^{-\frac{ \ell}{3}-100}$, then for $t \in I_{x,y} \setminus [x,y]$ we have \[ |\partial_t g(t)|=c \left| \lambda_x |t-x|^{c-1}-\lambda_{y} |t-y|^{c-1} \right| \simeq c(c-1)\xi^{c-2}  \left | \lambda_x^{\frac{1}{c-1}}|t-x|-\lambda_{y}^{\frac{1}{c-1}}|t-y| \right |\text{,}
\] for some $\xi$ between $\lambda_{y}^{\frac{1}{c-1}}|t-y|$ and $ \lambda_{x}^{\frac{1}{c-1}}|t-x|$. We also used the fact that in $I_{x,y} \setminus [x,y]$, $t-x$ and $t-y$ have the same sign. We also have  \[ \begin{split} 
			|\partial_t g(t)| &\gtrsim (c-1)\xi^{c-2} | \lambda_{x}^{\frac{1}{c-1}}(t-x)-\lambda_y^{\frac{1}{c-1}}(t-y)| \gtrsim (c-1)\lambda_x 2^{cj_x} 2^{-2j_x} | t-x- \left(\frac{\lambda_y}{\lambda_x}\right)^{\frac{1}{c-1}}(t-y)| \\ &  \gtrsim (c-1)2^{c \ell-2j_x }|t-x-(1+\eps(x,y))(t-y)| \gtrsim(c-1)  2^{c \ell-2j_x } 2^{j_x-\frac{\ell}{3} } \gtrsim (c-1)2^{(c-\frac{1}{3})\ell-j_x} \text{.}
		\end{split}  \] 
Van der Corput's lemma yields
\begin{multline}
\left|\int_{I_{x,y} \setminus [x,y]} e(\lambda_x[x-t]_{\mathsf{i}}^c-\lambda_y[y-t]_{\mathsf{i}}^c) \psi^{\tau}_{j_x}(x-t) \psi^{\tau}_{j_y}(y-t) \d t \right|
\\
\lesssim  (c-1)^{-1}\frac{2^{-2j_x}}{2^{-j_x+(c-\frac{1}{3})\ell}}=(c-1)^{-1}2^{-(c-\frac{1}{3})\ell-j_x}\le(c-1)^{-1}2^{-\frac{1}{3}\ell-j_x}\text{,}
\end{multline}
 which is a bound of the desired form.
		
		Finally, if $\eps(x,y) > 2^{-\frac{\ell}{3}-100}$ by recasting the above calculation we have that 
		\begin{multline}
			E_{x,y} \coloneqq \left\{ t \in I_{x,y} \setminus [x,y]: |\partial_t g(t)| \leq 2^{-j_x+\frac{\ell}{3}} \right\}
			\\
			\subseteq \left\{t \in I_{x,y} \setminus [x,y]:  2^{c \ell -2j_x}|y-x+\eps(x,y)(t-y)| \lesssim \big(c(c-1)\big)^{-1}2^{-j_x+\frac{\ell}{3}} \right\}
			\\
			\subseteq \left\{ t \in I_{x,y} \setminus [x,y]:  |y-x+\eps(x,y)(t-y)| \lesssim(c-1)^{-1} 2^{j_x+(-c+\frac{1}{3})\ell} \right\}\text{,}
		\end{multline}
		which yields that
$ |E_{x,y}| \lesssim (c-1)^{-1}2^{j_x+(-c+\frac{2}{3})\ell}\lesssim(c-1)^{-1}2^{j_x-\frac{1}{3}\ell}$. Therefore we can split the oscillatory integral on $I_{x,y} \setminus [x,y]$ as follows  \begin{multline}
\left|\int_{I_{x,y} \setminus [x,y]} e(\lambda_x[x-t]_{\mathsf{i}}^c-\lambda_y[y-t]_{\mathsf{i}}^c) \psi^{\tau}_{j_x}(x-t) \psi^{\tau}_{j_y}(y-t) \d t \right| \leq \left|\int_{E_{x,y}}\cdots \right|+\left| \int_{I_{x,y} \setminus E_{x,y}}\cdots \right| 
\\
\lesssim (c-1)^{-1} 2^{-2j_x+j_x-\frac{\ell}{3}}+(c-1)^{-1}\frac{2^{-2j_x}}{2^{-j_x+\frac{\ell}{3}}}
\lesssim (c-1)^{-1}2^{-j_x-\frac{1}{3}\ell}\text{,}\quad\text{as desired.}
\end{multline}
 The proof is complete.
\end{proof}
Let us note that with this lemma essentially yields Theorem~$\ref{Guo}$ on $L^2$ for $c\in(1,2)$. Since we would like to keep track of the dependencies of constants with respect to $c\in(1,2)$ on $L^2$ we quickly showcase how this is done while establishing the bound $C_{2,\varepsilon,\mathsf{i}}\lesssim (c-1)^{-\frac{1}{2}}$ we claimed earlier, see Theorem~$\ref{Guo}$ above.
\begin{lemma}\label{GuoConstant}
For every $c\in(1,2)$ and $\mathsf{i}\in\{0,1\}$ we have
\[
\big\|\mathcal{C}^{\mathsf{i}}_{c}f\big\|_{L^2(\mathbb{R})}\lesssim(c-1)^{-\frac{1}{2}}\|f\|_{L^2(\mathbb{R})}\text{,}
\]
where the operator $\mathcal{C}^{\mathsf{i}}_{c}$ is defined in \eqref{GuoOperator} and the implicit constant is absolute.
\end{lemma}
\begin{proof}
It suffices to establish the estimate for smooth compactly supported functions and we note that
\begin{multline}\label{splitforthecontinuous}
\int_{\mathbb{R}}f(x-t) \frac{e(\lambda [t]_{\mathsf{i}}^{c})}{t}\d t=\sum_{j \in \Z} \int_{\mathbb{R}}f(x-t) e(\lambda [t]_{\mathsf{i}}^{c}) \psi_j(t)\d t  
\\
=\sum_{j \in \Z} \sum_{r \in \Z} \cic{1}_{[2^r,2^{r+1})}(|\lambda|^{\frac{1}{c}}2^j ) \int_{\mathbb{R}}f(x-t) e(\lambda [t]_{\mathsf{i}}^{c}) \psi_j(t)\d t
=\sum_{r \in \Z} \int_{\R}f(x-t)e(\lambda \left[t\right]_{\mathsf{i}}^c) \psi_{j(r,\lambda)}(t) \d t
\\
= \bigg(\sum_{r> 0}+ \sum_{r\le 0}\bigg) \int_{\R}f(x-t)e(\lambda \left[t\right]_{\mathsf{i}}^c) \psi_{j(r,\lambda)}(t) \d t \coloneqq \mathsf{I}+\mathsf{II}\text{,}
\end{multline}
where $j(r,\lambda)$ is the unique $j\in\mathbb{Z}$ such that $2^r\le |\lambda|^{\frac{1}{c}}2^j<2^{r+1}$, i.e.:
\[
j(r,\lambda)= \big\lfloor\log_2 \big(2^r |\lambda|^{-\frac{1}{c}} \big)\big \rfloor= \big\lfloor r+\log_2\big(|\lambda|^{-\frac{1}{c}}\big)\big\rfloor=r+\big\lfloor \log_2\big(|\lambda|^{-\frac{1}{c}}\big)\big\rfloor\text{.}
\]

For the first summand in the last line of \eqref{splitforthecontinuous} we  have the following pointwise bound
\[
|\mathsf{I}| \le \sum_{r\in\mathbb{N}} \sup_{j\in\mathbb{Z}} \sup_{2^r\le|\lambda|^{\frac{1}{c}}2^{j}\le  2^{r+1}} \big|\mathcal{F}_{\mathbb{R}}^{-1}\big[\mathcal{F}_{\mathbb{R}}(f)(\cdot) H_j^{\mathsf{i}}(\cdot,\lambda)\big](x)\big|\text{,}
\]
and by Lemma~$\ref{largeelllp}$, we get that
\begin{equation}\label{Est2cont}
\|\mathsf{I}\|_{L^{2}(\mathbb{R})}\lesssim \sum_{r\in\mathbb{N}}(c-1)^{-\frac{1}{2}}2^{-\frac{r}{6}}\|f\|_{L^2(\mathbb{R})}\lesssim (c-1)^{-\frac{1}{2}}\|f\|_{L^2(\mathbb{R})}\text{.}
\end{equation}

It remains to treat the second summand, which we bound pointwise using the Hardy--Littlewood maximal function and maximal truncations of the Hilbert tranform. Firstly, note that 
\[ \left| \mathsf{II} \right|= \sum_{r \leq 0} \int_{\R}f(x-t)\left(e(\lambda \left[t\right]_{\mathsf{i}}^c )-1\right)\psi_{j(r,\lambda)}(t) \d t +\sum_{r \leq 0 } \int_{\R}f(x-t)\psi_{j(r,\lambda)}(t) \d t\text{.} 
\]
On the one hand, we have
\[ \Big|\sum_{r \leq 0} \int_{\R}f(x-t)\left(e(\lambda \left[t\right]_{\mathsf{i}}^c )-1\right)\psi_{j(r,\lambda)}(t) \d t \Big| \leq \sum_{r \leq 0} 2^{rc} \mathrm{M}f(x) \lesssim \mathrm{M}f(x)\text{,}
\]
and on the other hand, we have
\[
\sum_{r \leq 0 } \int_{\R}f(x-t)\psi_{j(r,\lambda)}(t) \d t=\sum_{j \leq \lfloor \log_2(|\lambda|^{-\frac{1}{c}})\rfloor  } \int_{\R}f(x-t) \psi_{j}(t) \d t \leq \mathcal{H}^*f(x)\text{,} 
\]
where $\mathcal{H}^*f(x)= \sup_{k \in \Z} \Big|\sum_{j \leq k} \int_{\R}f(x-t) \psi_j(t) \d t\Big|$. Since both $ \mathrm{M}$ and $ \mathcal{H}^*$ are bounded operators on $L^2$, we immediately get that $\|\mathsf{II}\|_{L^2(\mathbb{R})}\lesssim \|f\|_{L^2(\mathbb{R})}$.

Taking the above estimate into account together with \eqref{splitforthecontinuous} and \eqref{Est2cont} concludes the proof.
\end{proof}
\,\\
\textbf{Estimates for $S_3$.} Note that
	\begin{multline}
	\Big|	\mathcal{F}_{\mathbb{R}}^{-1}\Big[\sum_{j \leq 0} \mathcal{F}_{\mathbb{R}}[f](\cdot) H^{\mathsf{i}}_j(\cdot,\lambda) \Big](x)\Big| \leq  \Big|  \sum_{j\le 0}\int_{\mathbb{R}}(f*\mathcal{F}^{-1}_{\mathbb{R}}[\eta_1])(x-t)e(\lambda[t]_{\mathsf{i}}^c)\psi_j(t)\d t  \Big|
		\\
		\leq \Big|\sum_{j\le 0}\int_{\mathbb{R}}(f*\mathcal{F}^{-1}_{\mathbb{R}}[\eta_1])(x-t)\big(e(\lambda[t]_{\mathsf{i}}^c)-1\big)\psi_j(t)\d t \Big|+\Big|\sum_{j\le 0}\int_{\mathbb{R}}(f*\mathcal{F}^{-1}_{\mathbb{R}}[\eta_1])(x-t)\psi_j(t)\d t \Big| \text{.}
	\end{multline}
	For the first summand we note that for every $x\in\mathbb{R}$ we have
	\[
	\begin{split}
		&    \sup_{\lambda\in[-1/2,1/2)}\Big|\sum_{j\le 0}\int_{\mathbb{R}}(f* \mathcal{F}_{\mathbb{R}}^{-1}[\eta_1])(x-t)\big(e(\lambda [t]_{\mathsf{i}}^c)-1\big)\psi_j(t)\d t \Big|
		\\
		&\lesssim\sup_{\lambda\in[-1/2,1/2)}  \sum_{j\le 0} \int_{\mathbb{R}}|f* \mathcal{F}_{\mathbb{R}}^{-1}[\eta_1]|(x-t)|\lambda| |t|^c|\psi_j(t)|\d t
		\\
		&\lesssim\sum_{j \leq 0}2^{jc}\int_{\mathbb{R}}|f* \mathcal{F}_{\mathbb{R}}^{-1}[\eta_1]|(x-t)|\psi_j(t)| \d t \le\sum_{j \le 0} 2^{jc} M\big(f* \mathcal{F}_{\mathbb{R}}^{-1}[\eta_1]\big)(x)\lesssim M\big(f* \mathcal{F}_{\mathbb{R}}^{-1}[\eta_1]\big)(x)\text{,}
	\end{split}
	\]
	and one may use this pointwise estimate and the $L^p$-boundedness of the Hardy--Littlewood maximal operator to conclude.
	
	For the second term we note that one may control it by the maximal truncations of the Hilbert transform and therefore conclude in a similar manner. This is standard and we chose not to give full details here. We briefly mention that if we let
	\[
	\mathcal{H}_0 f(x)\coloneqq\sum_{j \le 0}\int_{\mathbb{R}}f(t)\psi_j(x-t)\d t
	\]
	and $\mathcal{H}$ be the Hilbert transform, by Cotlar's inequality we can see that
	\[
	|\mathcal{H}_0 f(x)| \lesssim M f(x)+M(\mathcal{H}f)(x)
	\]
	and conclude, see for example the proof of Lemma~6.20 in \cite{ParissisHa}. The estimate for $S_3$ is complete and so is the proof of Proposition~$\ref{ContRedExpl}$.
\end{proof}
Let us mention that by combining Lemma~$\ref{largeelllp}$, Lemma~$\ref{GuoConstant}$, as well as the fact that our handling of $S_3$ clearly yields constants independent of $c$ we get
\begin{equation}\label{L2ContEst}
\bigg\|\sup_{\lambda\in [-1/2,1/2)}\Big|\mathcal{F}^{-1}_{\mathbb{R}}\Big[\sum_{j\in\mathbb{N}}\mathcal{F}_{\mathbb{R}}[f](\cdot)H_j^{\mathsf{i}}(\cdot,\lambda)\cic{1}_{X_j}(\lambda)\Big]\Big|\bigg\|_{L^2(\mathbb{R})}\lesssim (c-1)^{-\frac{1}{2}}\|f\|_{L^2(\mathbb{R})}\text{,}
\end{equation}
and returning to \eqref{Sec5absconstant}, we get that
\begin{equation}\label{FinConstantEst}
\sup_{f\in\ell^2:\,\|f\|_{\ell^2}\le 1}\bigg\|\sup_{\lambda \in [-1/2,1/2)}\Big| \sum_{n \neq 0}f(x-n) \frac{e(\lambda \mathsf{sign}^{\mathsf{i}}(n) \lfloor |n|^{c} \rfloor)}{n}\Big|\bigg\|_{{\ell_{\d x}^2}}
\lesssim (c-1)^{-\frac{1}{2}}(2-c)^{-\frac{13}{4}}\text{,}\quad\text{as claimed earlier.}
\end{equation}

\end{document}